\documentclass[a4paper,11pt]{amsart}

\usepackage{a4wide, amsmath, amsfonts, amssymb, mathrsfs, amsthm, bbm, color, stmaryrd, breqn}
\usepackage[hyperfootnotes=false]{hyperref} 

\allowdisplaybreaks[2]

\numberwithin{equation}{section}

\newtheorem{theorem}{Theorem}[section]
\newtheorem{lemma}[theorem]{Lemma}
\newtheorem{corollary}[theorem]{Corollary}
\newtheorem{remark}[theorem]{Remark}
\newtheorem{proposition}[theorem]{Proposition}

\theoremstyle{definition}
\newtheorem{definition}[theorem]{Definition}

\newtheorem{problem}[theorem]{Optimal Extension Problem}

\newcommand{\dd}{\,\mathrm{d}}
\renewcommand{\d}{\mathrm{d}}
\newcommand{\R}{\mathbb{R}}
\newcommand{\N}{\mathbb{N}}

\newcommand{\X}{{\bf X}}
\newcommand{\Y}{{\bf Y}}
\newcommand{\Z}{{\bf Z}}
\newcommand{\var}{\text{-}\mathsf{var}}
\newcommand{\vertiii}[1]{{\left\vert\kern-0.25ex\left\vert\kern-0.25ex\left\vert #1 \right\vert\kern-0.25ex\right\vert\kern-0.25ex\right\vert}}

\title{Optimal extension to Sobolev rough paths}

\author[Liu]{Chong Liu}
\address{Chong Liu, University of Oxford, United Kingdom}
\email{chong.liu@maths.ox.ac.uk}

\author[Pr{\"o}mel]{David J. Pr{\"o}mel}
\address{David J. Pr{\"o}mel, Universit{\"a}t Mannheim, Germany}
\email{proemel@uni-mannheim.de}

\author[Teichmann]{Josef Teichmann}
\address{Josef Teichmann, Eidgen{\"o}ssische Technische Hochschule Z{\"u}rich, Switzerland}
\email{josef.teichmann@math.ethz.ch}

\date{\today}

\begin{document}

\begin{abstract}
  We show that every $\R^d$-valued Sobolev path with regularity~$\alpha$ and integrability~$p$ can be lifted to a weakly geometric rough path in the sense of T.~Lyons with exactly the same regularity and integrability, provided $\alpha >1/p>0$. Moreover, we prove the existence of unique rough path lifts which are optimal w.r.t. strictly convex functionals among all possible rough path lifts given a Sobolev path. This paves a way towards classifying rough path lifts as solutions of optimization problems. As examples, we consider the rough path lift with minimal Sobolev norm and characterize the Stratonovich rough path lift of a Brownian motion as optimal lift w.r.t. a suitable convex functional. Generalizations of the results to Besov spaces are briefly discussed. 
\end{abstract}

\maketitle




\section{Introduction}

The original motivation and central aim of rough path theory is the study of controlled ordinary differential equations
\begin{equation}\label{eq:intro equation}
  \d Y_t = V(Y_t)\dd X_t, \quad Y_0 = y_0, \quad t \in [0,T],
\end{equation}
where $(X_{t})_{t\in [0,T]}$ is a continuous $d$-dimensional driving signal, $y_0 \in \R^e$ is an initial value and $V$ is a smooth map from $ \R^d $ to the space of endomorphisms of $\R^e$. Ordinary differential equations of this type are classical objects as long as the driving signal is at least weakly differentiable with $p$-integrable derivative, that is, $X$ belongs to some Sobolev space $W^{1}_p$. However, it is a delicate problem to extend the solution map $X \mapsto Y$ in a meaningful way to larger spaces of driving signals containing, e.g., sample paths of frequently considered stochastic processes like Brownian motion, see \cite{Lyons1991}.

In order to set up a deterministic solution theory consistent with the classical one but covering many interesting examples of stochastic processes, T.~Lyons~\cite{Lyons1998} realized that the driving signal~$X$ needs not only to take values in $\R^d$ but instead in the step-$N$ free nilpotent group $G^N(\mathbb{R}^d)$ (see Subsection~\ref{subsec:Sobolev rough path} for all necessary details), which equals by the Chow--Rashevskii theorem (see, e.g., \cite{Gromov1996}) to the values of step-$N$ signatures of paths with bounded variation, i.e.,
\begin{equation*}
  G^N(\mathbb{R}^d)= \{S_N(Z)_{0,T} \,:\, Z\in C^{1\var}([0,T];\mathbb{R}^d)\},\quad N\in\mathbb{N},
\end{equation*}
where the step-$N$ signature of a path~$Z$ of bounded variation is given by 
\begin{align*}
  S_N(Z)_{s,t}:=&\bigg (1, \int_{s<u<t}\dd Z_u, \dots, \int_{s<u_1<\dots <u_N<t}\dd Z_{u_1} \otimes \cdots \otimes \d Z_{u_N} \bigg).
\end{align*}
This led to the notion of rough paths: $\X\colon [0,T]\to G^N(\mathbb{R}^d)$ is called a (weakly geometric) rough path if $\X$ is $\alpha$-H\"older continuous or of finite $1/\alpha$-variation, for $\alpha >1/N$. Assuming that the driving signal~$\X$ of the differential equation~\eqref{eq:intro equation} is a rough path, the theory of rough paths initiated by T.~Lyons establishes that~\eqref{eq:intro equation} possesses a unique solution~$Y$ and that the solution map $\X\mapsto Y$ is locally Lipschitz continuous. In the area of stochastic analysis this solution map is then often called It{\^o}--Lyons map. For more details about rough path theory we refer the interested reader to the introductory textbooks~\cite{Lyons2007,Lejay2009,Friz2010,Friz2014}.

This immediately raises the question whether every $\R^d$-valued path~$X$ can be lifted to a weakly geometric rough path~$\X$ in the sense that the projection of $\X$ onto the path-level is $X$. Of course, for sufficiently regular paths this can easily be achieved by, e.g., Young integration \cite{Young1936}, but in general this question becomes rather challenging. The first affirmative answer was given by Lyons and Victoir~\cite{Lyons2007a}. They show, in particular, that an $\R^d$-valued H{\"o}lder continuous path can always be lifted to a H{\"o}lder continuous rough path, which then immediately extends to $p$-variation by a re-parameterization argument. While the approach of Lyons and Victoir is non-constructive, an explicit approach based on so-called Fourier normal ordering was developed by J.~Unterberger~\cite{Unterberger2010}. The latter approach is in a related spirit to the one relying on Hairer's regularity structures~\cite[Section~13]{Friz2014}, see also \cite{Brault2019,Liu2021b}. Recently, Tapia and Zambotti~\cite{Tapia2020} generalized Lyons--Victoir extension theorem to the case of anisotropic H{\"o}lder continuous paths, i.e., allowing each component to have a different regularity. They provide a constructive version of Lyons--Victoir approach by using an explicit form of the Baker--Campbell--Hausdorff formula. Applications of these extension theorems for rough paths can be found in, e.g., \cite{Qian2011,Cass2015}.
 
In this article we shall consider weakly geometric rough paths with respect to a fractional Sobolev topology with regularity~$\alpha$ and integrability~$p$. This seems to be a very natural choice as Sobolev spaces provide a very successful framework to deal with (classical) ordinary and partial differential equations. It turns out that the notion of so-called Sobolev rough paths (Definition~\ref{def:Sobolev geometric rough path}) is not only a feasible object in the context of rough path theory, see~\cite{Liu2021}, but also offers many favourable properties, which are not provided by the commonly used distances on the rough path spaces such as H{\"o}lder or $p$-variation. For instance, let us recall that the classical Sobolev spaces for $p<\infty$ are known to be strictly convex Banach spaces. Furthermore, note that Sobolev spaces cover the H{\"o}lder spaces as a special case by setting $p=\infty$.

Our first contribution is to prove that every path of Sobolev regularity~$\alpha$ and integrability~$p$ can be lifted to a Sobolev rough path provided $1/p <\alpha <1$, which represents a Lyons--Victoirs extension theorem for Sobolev paths with arbitrary low regularity. For this purpose, our approach combines ideas from \cite{Lyons2007a} with a discrete characterization of fractional Sobolev spaces as recently provided by~\cite{Liu2020}. Furthermore, our Sobolev extension theorem provides the existence of a joint Sobolev rough path extension given, for example, two Sobolev rough paths taking values in lower dimensional free nilpotent groups.    

Besides its originally non-constructive nature, a second concern with the Lyons--Victoir extension theorem is the fact that it does of course not provide a canonical way to lift a path. It is even well-known that the existence of one rough path lift ensures the existence of infinitely many rough path lifts and thus one is left with the task to select in some way a canonical one. The usual way to circumvent this issue is to keep (if possible) the probabilistic nature of the driving signal~$X$ in mind. Then, it is possible to select a canonical rough path lift based on some type of stochastic integration, see, e.g., for fractional Brownian motions~\cite{Coutin2002} or for martingales~\cite{Coutin2005}. 

Our second contribution is, using the present Sobolev stetting, to provide the existence of unique rough path lifts which are optimal with respect to a deterministic selection criterion. Let us emphasize that here the strictly convex nature of the Sobolev topology seems to be essential. As a first example, we show that there always exists a unique rough path lift of ``minimal length'' for every given path~$X$ with Sobolev regularity $\alpha \in (1/3,1/2)$, in the sense that there exists a unique rough path lift possessing the minimal Sobolev norm among all rough path lifts of~$X$. As a second example, we prove that the Stratonovich rough path lift of a Brownian motion is almost surely the unique optimal rough path lift with respect to a suitable selection criterion.  

\smallskip
\noindent{\bf Organization of the paper:} In Section~\ref{sec:lifting} the rough path lift of a path with sufficient Sobolev regularity is constructed. In Section~\ref{sec:optimal lift} the existence of a unique optimal rough path lift is proven and examples of optimal rough path lifts are presented. The generalizations of some results to Besov spaces are discussed in Section~\ref{sec:Besov spaces}.

\smallskip
\noindent{\bf Acknowledgment:} C.~Liu and J.~Teichmann gratefully acknowledge support by the ETH foundation. C.~Liu was employed at ETH Zurich when this project was commenced. D.J. Pr\"omel is grateful to Martin Huesmann for inspiring discussions about the problem of lifting a path in a ``optimal'' manner. D.J.~Pr{\"o}mel and J.~Teichmann gratefully acknowledge support by the SNF Project 163014.

\subsection{Basic notation and function spaces}

As usual, $\R$ and $\mathbb{Z}$ are the real numbers resp. the integers, $\N :=\{1,2,\dots \}$ are the natural numbers and we set $\N_0:=\N \cup \{0\}$. For two real functions $a,b$ depending on variables~$x$ one writes $a\lesssim b$ or $a\lesssim_z b$ if there exists a constant $C(z)>0$ such that $a(x) \leq C(z)\cdot b(x)$ for all $x$, and $a\sim b$ if $a\lesssim b$ and $b\lesssim a$ hold simultaneously. A partition~$\pi$ of an interval $[0,T]$ is a collection of finitely many essentially disjoint interval covering $[0,T]$, i.e., $\pi := \{ [t_{k-1},t_{k}]\,:\,0=t_0<t_1<\dots<t_n=T,\, n\in \N\}$. 

Let $(E,d)$ be a metric space and $p\in [1,+\infty)$. The space of all continuous functions $f\colon [0,T]\to E$ is denoted by $C([0,T];E)$. We can define a metric thereon by $ d_\infty(f,g) := \sup_{0\leq t\leq T}d(f(t),g(t)) $. If $ E $ is normed vector space, we set $\|f\|_{\infty}:=\sup_{0\leq t\leq T}\|f(t)\|$. The space $C^{p\var}([0,T];E)$ consists of all continuous functions $f\colon [0,T]\to E$ of finite $p$-variation, i.e.,
\begin{equation*}
  \| f\|_{p\var ;[0,T]}:= \sup_{\pi \subset [0,T]}\bigg( \sum_{[s,t]\in \pi} d(f(s),f(t))^p \bigg)^{1/p}<+\infty,
\end{equation*}
where the supremum is taken over all partitions~$\pi$ of the interval $[0,T]$. For $\alpha \in (0,1)$ and $p\in (1,+\infty)$ the space $W^{\alpha}_p([0,T];E)$ consists of all measurable functions $f\colon [0,T]\to E$ such that
\begin{equation*}
  \iint_{[0,T]^2} \frac{d(f(s),f(t))^p}{|t-s|^{\alpha p+1}} \dd s \dd t<+\infty.
\end{equation*}

\section{Lifting Sobolev paths to Sobolev rough paths}\label{sec:lifting}

This section is devoted to show that every path of suitable Sobolev regularity can be lifted to a weakly geometric rough path \emph{possessing exactly the same Sobolev regularity}, see in particular Subsection~\ref{subsec:Lyons--Victoir lift}.

\subsection{Sobolev rough path}\label{subsec:Sobolev rough path}

Let us start by introducing the notion of Sobolev rough paths and by fixing the basic definitions of rough path theory, following the commonly used notation as, e.g., introduced in~\cite{Friz2010} or \cite{Lyons2007a}. 

Let $\R^d$ be the Euclidean space with norm~$| \cdot |$ for $d\in \N$. The tensor algebra over $\R^d$ is denoted by $T(\R^d) := \bigoplus_{n=0}^\infty (\R^d)^{\otimes n}$ where $\big(\mathbb{R}^d\big)^{\otimes n}$ stands for the $n$-tensor space of $\mathbb{R}^d$ and where we use the convention $(\R^d)^{\otimes 0}:=\mathbb{R}$. $T(\R^d)$ is equipped with the standard addition $+$, tensor multiplication~$\otimes$ and scalar product. We consider it as \emph{a} representation of the free algebra with $d$ indeterminates.

We recall further some group theoretic prerequisites. For any $N \in \N_0$, $T^{N}(\R^d)$ denotes the quotient algebra of $T(\R^d)$ by the ideal $\bigoplus_{m=N+1}^\infty (\R^d)^{\otimes m}$ with the corresponding algebraic structures making it a free $N$-step nilpotent algebra with $d$ indeterminates. On $T^{N}(\R^d)$ such as on $T(\R^d)$ one can define a Lie bracket by the commutator formula
$$
  [a,b] := a\otimes b - b \otimes a,
$$
which makes $T^{N}(\R^d)$ into a Lie algebra. Let $\mathcal{G}^{N}(\R^d)$ be the Lie subalgebra of $T^{N}(\R^d)$ generated by elements in $\R^d$. Note that
\begin{equation*}
  \mathcal{G}^{N}(\R^d) := \bigoplus_{i=1}^n V_i,
\end{equation*}
where $V_1 := \R^d$ and $V_{i+1} := [\R^d,V_i]$. $\mathcal{G}^{N}(\R^d)$ is called the free nilpotent Lie algebra of step~$N$. The exponential, logarithm and inverse function are defined on $T^{(N)}(\R^d)$ by means of their power series. We denote by $G^{N}(\R^d) := \exp(\mathcal{G}^{N}(\R^d))$, which is a connected nilpotent Lie group with the group operator $\otimes$. By construction, $G^{N}(\R^d)$ is the Carnot group with Lie algebra $\mathcal{G}^N(\R^d)$. On $G^N(\mathbb{R}^d)$ usually two types of complete metrics are considered, which generate the standard trace topology inherited from $ T^N(\mathbb{R}^d)$: the first metric is defined by
\begin{equation*}
  \rho(g,h) := \max_{i=1,\dots,N}|\pi_i(g - h)|\quad \text{for} \quad g,h \in G^N(\R^d),
\end{equation*}
where $\pi_i$ denotes the projection from $\bigoplus_{i=0}^N (\R^d)^{\otimes i}$ onto the $i$-th level. We set $|g|:=\rho(g, 1)$ for $g \in G^N(\R^d)$. The second one is the Carnot--Caratheodory metric $d_{cc}$, which is given by
$$
  d_{cc}(g,h) := \|g^{-1} \otimes h\|  \quad \text{for} \quad  g,h \in G^N(\R^d),
$$
where $\| \cdot \| $ is the Carnot--Caratheodory norm defined via \cite[Theorem~7.32]{Friz2010}, cf. \cite[Definition~7.41]{Friz2010}. These two metrics are in general not equivalent (unless $d = 1$) in the sense that there exist constants $C_1,C_2$ such that $C_1 \rho(g,h) \leq d_{cc}(g,h) \leq C_2 \rho(g,h)$ for all $g,h \in G^N(\R^d)$. Instead, it holds that 
$$
  \rho(g,h) \lesssim d_{cc}(g,h) \quad \text{and} \quad  d_{cc}(g,h) \lesssim \rho(g,h)^{1/N}
$$ 
uniformly on bounded sets (w.r.t.~the Carnot--Caratheodory norm), see \cite[Proposition~7.49]{Friz2010}. For more information regarding $G^N(\R^d)$ we refer to \cite[Chapter~7]{Friz2010}. Notice that we do \emph{not} consider the Carnot--Caratheodory metric in the sense of sub-Riemannian geometry in this article.

For $N\in\mathbb{N}$ and a path $Z\in C^{1\var}([0,T];\mathbb{R}^d)$, its $N$-step signature is given by 
\begin{align*}
  S_N(Z)_{s,t}:=&\bigg (1, \int_{s<u<t}\dd Z_u, \dots, \int_{s<u_1<\dots <u_N<t}\dd Z_{u_1} \otimes \cdots \otimes \d Z_{u_N} \bigg) \\
  &\in T^N(\mathbb{R}^d):= \bigoplus_{k=0}^N \big(\mathbb{R}^d\big)^{\otimes k}\subset T(\R^d),
\end{align*}
cf. \cite[Definition~7.2]{Friz2010}. As we have mentioned in the introduction, the corresponding space of endpoints at \emph{one fixed} point in time $T$ of all these lifted paths equals the step-$N$ free nilpotent group (w.r.t. $\otimes$) by the Chow--Rashevskii theorem:
\begin{equation*}
  G^N(\mathbb{R}^d)= \{S_N(Z)_{0,T} \,:\, Z\in C^{1\var}([0,T];\mathbb{R}^d)\}\subset T^N(\mathbb{R}^d).
\end{equation*}

If $\X$ defined on $[0,T]$ is a path taking values in $G^N(\R^d)$, we set $\X_{s,t} := \X_s^{-1} \otimes \X_t$ for any subinterval $[s,t] \subset [0,T]$. For $r \in \R_{+}$, we define 
$$
  [r] := \sup\{n \in \mathbb{Z}\, :\, n \leq r\}\quad \text{and}\quad\lfloor r \rfloor := \sup\{n \in \mathbb{Z}\, :\, n < r\}.
$$
In this article we shall always equip the free nilpotent Lie group $G^N(\R^d)$ with the Carnot--Caratheodory metric~$d_{cc}$, which gives then a metric space. This allows for defining the \textit{fractional Sobolev (semi)-norm} for a path $\X\colon [0,T] \to G^N(\R^d)$ by 
\begin{equation*}
  \| \X \|_{W^{\alpha}_p} := \Big(\iint_{[0,T]^2}\frac{d_{cc}(\X_s,\X_t)^p}{|t-s|^{\alpha p + 1}} \dd s \dd t\Big)^{\frac{1}{p}}
\end{equation*}
for $\alpha \in (0,1)$ and $p\in (1,+\infty)$. Note that, for a continuous path $\X\colon [0,T] \to G^N(\R^d)$ and $T=1$, the fractional Sobolev (semi)-norm can be equivalently defined in a discrete way by 
\begin{equation}\label{eq:discrete Sobolev norm}
  {\|\X\|}_{W^{\alpha}_p,(1)} :=\bigg( \sum_{j \geq 0} 2^{j( \alpha p - 1)} \sum_{m=0}^{2^j-1}  d_{cc} \big(\X_{ \frac{m}{2^j}},\X_{\frac{m+1}{2^j}} \big)^p \bigg)^{1/p} <+\infty,
\end{equation}
see \cite[Theorem~2.2]{Liu2020}. Furthermore, we use
\begin{equation*}
  | \X |_{W^{\alpha}_p} := \Big(\iint_{[0,T]^2}\frac{\rho(\X_s,\X_t)^p}{|t-s|^{\alpha p + 1}} \dd s \dd t\Big)^{\frac{1}{p}}
\end{equation*}
for $\alpha \in (0,1)$ and $p\in (1,+\infty)$. The Sobolev topology leads naturally to the notion of (fractional) Sobolev rough paths.

\begin{definition}[Sobolev rough path]\label{def:Sobolev geometric rough path}
  Let $\alpha \in (0,1)$ and $p\in (1,+\infty)$ be such that $\alpha > 1/p$. The space $W^{\alpha}_p([0,T];G^{[\frac{1}{\alpha}]}(\mathbb{R}^d))$ consists of all paths $\X\colon[0,T]\to G^{[\frac{1}{\alpha}]}(\mathbb{R}^d) $ such that $\| \X \|_{W^{\alpha}_p} <+\infty$, and is called \textit{weakly geometric Sobolev rough path space}. Every element $\X$ in the space $W^{\alpha}_p([0,T];G^{[\frac{1}{\alpha}]}(\mathbb{R}^d))$ is called a \textit{weakly geometric rough path of Sobolev regularity} $(\alpha,p)$ or in short \textit{Sobolev rough path}. We say that $\X\in W^{\alpha}_p([0,T];G^{[\frac{1}{\alpha}]}(\mathbb{R}^d))$ is a (weakly geometric) \textit{rough path lift} of~$X$ if $\pi_{1}(\X) = X$ for a $\R^d$-valued continuous path $X\colon [0,T]\to \R^d$.
\end{definition}

\begin{remark}
  With the parameters $\alpha \in (0,1)$ and $p\in (1,+\infty)$ such that $\alpha > 1/p$ every weakly geometric rough path of Sobolev regularity $(\alpha,p)$ is continuous. Indeed, an application of the Garsia--Rodemich--Rumsey inequality, see e.g.~\cite[Theorem~A.1]{Friz2010}, implies the existence of a constant $C > 0$ such that
  \begin{equation*}
    d_{cc}(\X_s,\X_t) \leq C {|t-s|}^{\alpha-\frac{1}{p}} 
  \end{equation*}
  for all $s,t \in [0,T]$ and all $\X \in W^{\alpha}_p([0,T];G^{[\frac{1}{\alpha}]}(\mathbb{R}^d))$. 
\end{remark}

In the sequel, $K$ stands always for a closed normal subgroup of $G^{N}(\R^d)$ with the corresponding Lie algebra $\mathcal{K} \subset \mathcal{G}^N(\R^d)$. Again, by $ \| \cdot \| $ we denote the Carnot--Caratheodory norm on $G^N(\R^d)$. The quotient Lie group $G^{N}(\R^d)/K$ is equipped with the quotient homogeneous norm 
\begin{align*}
  \| \cdot \|_{G^{N}(\R^d)/K} \colon &G^{N}(\R^d)/K \rightarrow \R, \\
  & gK \mapsto \inf_{k \in K}\|g\otimes k\| ,
\end{align*}
which defines a metric $d$ on $G^{N}(\R^d)/K$. The canonical homomorphism from $G^{N}(\R^d)$ onto $G^{N}(\R^d)/K$ is denoted by $\pi_{G^{N}(\R^d), G^{N}(\R^d)/K}$. For more details about Carnot groups and free nilpotent Lie group we refer the reader to \cite[Section~3]{Lyons2007a} and the references therein.

We say that a path $X\colon [0,T] \rightarrow G^N(\R^d)/K$ is of Sobolev regularity $(\alpha,p)$, in notation $X \in W^\alpha_p([0,T];G^N(\R^d)/K)$, if it satisfies that
$$
  \|X\|_{W^\alpha_p} := \Big(\iint_{[0,T]^2}\frac{d(X_s,X_t)^p}{|t-s|^{\alpha p + 1}} \dd s \dd t\Big)^{\frac{1}{p}} < +\infty,
$$
where $d$ is the metric induced by the quotient norm $\| \cdot\|_{G^N(\R^d)/K}$ defined as above. The norm $\|X\|_{W^\alpha_p,(1)}$ is defined as in Formula \eqref{eq:discrete Sobolev norm} by replacing $d_{cc}$ through $d$. By \cite[Theorem~2.2]{Liu2020} we know that $\|\cdot\|_{W^\alpha_p}$ and $\|\cdot\|_{W^\alpha_p,(1)}$ are equivalent norms on $ W^\alpha_p([0,T];G^N(\R^d)/K)$ for $T=1$.

\subsection{Lyons--Victoir extension theorem for Sobolev paths}\label{subsec:Lyons--Victoir lift}

To obtain the Lyons--Victoir extension theorem for paths of arbitrary low Sobolev regularity, we generalizes the original approach of Lyons and Victoir~\cite{Lyons2007a} using a discrete characterization of fractional Sobolev space provided in \cite{Liu2020}. Recall that Lyons and Victoir show how to lift a H{\"o}lder continuous path to a rough path, which then extends to path of finite $p$-variation by a re-parameterization argument. The key idea is to construct the rough path lift of the H\"older continuous path on the dyadic grid by an induction argument and then verify that this gives indeed rise to a H\"older continuous rough path. This last step is based on a simple discrete characterization of H\"older spaces, see \cite[Lemma~2]{Lyons2007a}, which surprisingly generalizes to fractional Sobolev spaces, see \cite[Theorem~2.2]{Liu2020}. 

The next lemma can be viewed as a generalization of \cite[Lemma~13]{Lyons2007a} from H{\"o}lder paths to Sobolev paths, which relies on the axiom of choice. 

\begin{lemma}\label{lem:Generalization of Lemma 13 in Lyons-Victoir}
  Let $(G,\|\cdot\|_G)$ be a normed Carnot group with graded Lie algebra 
  $$
    \mathcal{G} = W_1 \oplus W_2\oplus \dots \oplus W_n, \quad W_{i+1}:=[W_i,W_1],\quad \text{for } i=1,\dots,n-1,
  $$
  for a normed space $W_1$. Let $K$ be a closed subgroup of $\exp(W_n)$, which gives a normed Carnot group $(G/K, \|\cdot \|_{G/K})$. Let $0<\alpha<1$ and $1\le p \le + \infty$ be such that $\alpha> 1/p$, and $X$ be a continuous path belonging to $W^\alpha_{p}([0,T];(G/K, \|\cdot \|_{G/K}))$. Then, if $\alpha < 1/n$, there exists a $(G,\|\cdot\|_G)$-valued path $\tilde{X}$ such that $\tilde{X} \in W^\alpha_{p}([0,T];(G, \|\cdot \|_{G}))$ and $\pi_{G,G/K}(\tilde{X}) = X$.
\end{lemma}

While the construction of $\tilde{X}$ follows the same lines as in the proof of \cite[Lemma~13]{Lyons2007a}, to prove its Sobolev regularity, however, requires some substantial extra work. 

\begin{proof}
  W.l.o.g. we may assume that $T=1$.
  In order to construct the lifted path $\tilde{X}$, we will first construct its increments $\tilde{X}_{s,t}$, where $s,t$ are two adjacent dyadic numbers in $[0,1]$. Such a collection of $\tilde{X}_{s,t}$ will satisfies that 
  \begin{equation}\label{eq:discrete characterization of Lie group valued Besov functions}
    \sum_{j=0}^\infty 2^{j (\alpha p- 1)} \sum_{m=0}^{2^j-1}  \Big \| \tilde{X}_{\frac{m}{2^j},\frac{m+1}{2^j}} \Big \|_{G}^p < + \infty,
  \end{equation}
  and 
  \begin{equation}\label{eq:the projection condition}
    \pi_{G,G/K}(\tilde{X}_{s,t}) = X_{s,t}.
  \end{equation}
  Then, due to \cite[Theorem~2.2]{Liu2020}, by multiplying these increments and then extending to the whole interval $[0,1]$, we will get a continuous path $\tilde{X}$ such that $\tilde{X} \in W^\alpha_{p}([0,1];(G, \|\cdot \|_{G}))$ and $\pi_{G,G/K}(\tilde{X}) = X$.
 
  Following the same construction procedure as in the proof of \cite[Lemma~13]{Lyons2007a}, we define recursively for $m \in \N_0$ some elements $Y_{\frac{k}{2^m},\frac{k+1}{2^m}} \in K$, $k = 0,\dots,2^m-1$, and the elements $\tilde{X}_{\frac{k}{2^m},\frac{k+1}{2^m}}$ by the formula
  $$
    \tilde{X}_{\frac{k}{2^m},\frac{k+1}{2^m}} = i_{G/K,G}(X_{\frac{k}{2^m},\frac{k+1}{2^m}}) \otimes Y_{\frac{k}{2^m},\frac{k+1}{2^m}},
  $$
  where $i_{G/K,G}$ is the injection of \cite[Proposition~6]{Lyons2007a}. This ensures that $\pi_{G,G/K}(\tilde{X}) = X$. Note that the proof of \cite[Proposition~6]{Lyons2007a} requires the axiom of choice.
  
  Let $Y_{0,1} := \exp(0)$. Then, we assume that $Y_{\frac{k}{2^m},\frac{k+1}{2^m}}$ and $\tilde{X}_{\frac{k}{2^m},\frac{k+1}{2^m}}$ have been constructed for all $0 \le k \le 2^m-1$ and a fixed $m \in \N_0$, and we define the two elements $Y_{\frac{2k}{2^{m+1}},\frac{2k+1}{2^{m+1}}}$ and $Y_{\frac{2k+1}{2^{m+1}},\frac{2k+2}{2^{m+1}}}$ to be both equal, and equal to the inverse of
  $$
    \delta_{2^{-\frac{1}{n}}}\Big(i_{G/K,G}(X_{\frac{2k}{2^{m+1}},\frac{2k+1}{2^{m+1}}})\otimes i_{G/K,G}(X_{\frac{2k+1}{2^{m+1}},\frac{2k+2}{2^{m+1}}}) \otimes \tilde{X}^{-1}_{\frac{k}{2^m},\frac{k+1}{2^m}}\Big),
  $$
  where $\delta_{\cdot}$ denotes the dilation operator on $G$ as defined in \cite[Definition~4]{Lyons2007a}. Using the same reasoning as in the proof of \cite[Lemma~13]{Lyons2007a}, one can verify that $Y_{\frac{2k}{2^{m+1}},\frac{2k+1}{2^{m+1}}} = Y_{\frac{2k+1}{2^{m+1}},\frac{2k+2}{2^{m+1}}}$ are elements in $K$, and $\tilde{X}_{\frac{k}{2^m},\frac{k+1}{2^m}} = \tilde{X}_{\frac{2k}{2^{m+1}},\frac{2k+1}{2^{m+1}}} \otimes \tilde{X}_{\frac{2k+1}{2^{m+1}},\frac{2k+2}{2^{m+1}}}$.

  Now, we set
  $$
    a_m := 2^{m(\alpha-\frac{1}{p})}\Big(\sum_{k=0}^{2^m-1} \|Y_{\frac{k}{2^m},\frac{k+1}{2^m}}\|_G^p\Big)^{\frac{1}{p}}.
  $$ 
  Since $\| \cdot \|_G$ is a sub-additive homogeneous norm, using the bounds given in \cite[Proposition~6]{Lyons2007a} for the injection $i_{G/K,G}$ and Minkowski's inequality, we can check that (the constant~$C$ may vary from line to line, but it only depends on $n,\alpha,p$ and $q$):
  \begin{align*}
    &2^{\frac{1}{n}}2^{-(m+1)(\alpha -\frac{1}{p})}a_{m+1} = 2^{\frac{1}{n}}\Big(\sum_{k=0}^{2^{m+1}-1} \|Y_{\frac{k}{2^{m+1}},\frac{k+1}{2^{m+1}}}\|_G^p\Big)^{\frac{1}{p}} \\
    &\quad= 2^{\frac{1}{n}}\Big(\sum_{k=0}^{2^m-1}\|\delta_{2^{-\frac{1}{n}}}\Big(i_{G/K,G}(X_{\frac{2k}{2^{m+1}},\frac{2k+1}{2^{m+1}}})\otimes i_{G/K,G}(X_{\frac{2k+1}{2^{m+1}},\frac{2k+2}{2^{m+1}}}) \otimes \tilde{X}^{-1}_{\frac{k}{2^m},\frac{k+1}{2^m}}\Big) \|_G^p\Big)^{\frac{1}{p}} \\
    &\quad\le \Big(\sum_{k=0}^{2^m-1}\Big(\|i_{G/K,G}(X_{\frac{2k}{2^{m+1}},\frac{2k+1}{2^{m+1}}})\|_G +  \|i_{G/K,G}(X_{\frac{2k+1}{2^{m+1}},\frac{2k+2}{2^{m+1}}})\|_G + \|\tilde{X}_{\frac{k}{2^m},\frac{k+1}{2^m}}\|_G\Big)^p\Big)^{\frac{1}{p}} \\
    &\quad\le \Big(\sum_{k=0}^{2^m-1}\|\tilde{X}_{\frac{k}{2^{m}},\frac{k+1}{2^{m}}}\|_G^p\Big)^{\frac{1}{p}} + C\Big(\sum_{k=0}^{2^{m+1}-1}\|X_{\frac{k}{2^{m+1}},\frac{k+1}{2^{m+1}}}\|_{G/K}^p\Big)^{\frac{1}{p}},
  \end{align*}
  and 
  \begin{align*}
    \Big(\sum_{k=0}^{2^m-1}\|\tilde{X}_{\frac{k}{2^{m}},\frac{k+1}{2^{m}}}\|_G^p\Big)^{\frac{1}{p}} &\le C\Big(\sum_{k=0}^{2^{m}-1}\|X_{\frac{k}{2^{m}},\frac{k+1}{2^{m}}}\|_{G/K}^p\Big)^{\frac{1}{p}} + \Big(\sum_{k=0}^{2^m-1}\|Y_{\frac{k}{2^{m}},\frac{k+1}{2^{m}}}\|_G^p\Big)^{\frac{1}{p}} \\
    &= C\Big(\sum_{k=0}^{2^{m}-1}\|X_{\frac{k}{2^{m}},\frac{k+1}{2^{m}}}\|_{G/K}^p\Big)^{\frac{1}{p}} +2^{-m(\alpha-\frac{1}{p})}a_m.
  \end{align*}
  Therefore, we obtain that
  \begin{align*}
    2^{\frac{1}{n}}&2^{-(m+1)(\alpha-\frac{1}{p})}a_{m+1} \\
    &\le 2^{-m(\alpha-\frac{1}{p})}a_m + C\Big(\sum_{k=0}^{2^{m}-1}\|X_{\frac{k}{2^{m}},\frac{k+1}{2^{m}}}\|_{G/K}^p\Big)^{\frac{1}{p}} + C\Big(\sum_{k=0}^{2^{m+1}-1}\|X_{\frac{k}{2^{m+1}},\frac{k+1}{2^{m+1}}}\|_{G/K}^p\Big)^{\frac{1}{p}},
  \end{align*}
  which in turn implies that
  \begin{equation}\label{eq:bound for am}
    a_{m+1} \le 2^{\alpha-\frac{1}{n}-\frac{1}{p}}a_m + b_m,
  \end{equation}
  where 
  $$
    b_m:= C2^{m(\alpha -\frac{1}{p})}\Big(\sum_{k=0}^{2^{m}-1}\|X_{\frac{k}{2^{m}},\frac{k+1}{2^{m}}}\|_{G/K}^p\Big)^{\frac{1}{p}} + C2^{(m+1)(\alpha-\frac{1}{p})}\Big(\sum_{k=0}^{2^{m+1}-1}\|X_{\frac{k}{2^{m+1}},\frac{k+1}{2^{m+1}}}\|_{G/K}^p\Big)^{\frac{1}{p}}.
  $$ 
  Since $\alpha < \frac{1}{n}$ holds by assumption, $r:= \alpha -\frac{1}{n}-\frac{1}{p} < 0$. Moreover, since $X$ is an element in $W^\alpha_{p}([0,1];(G/K, \|\cdot\|_{G/K}))$, by \cite[Theorem~2.2]{Liu2020} we have 
  $
    \sum_{m=0}^\infty b_m^p \lesssim \|X\|_{W^\alpha_p,(1)}  < + \infty.
  $
  Iterating applications of inequality~\eqref{eq:bound for am}, we see that
  $$
    a_{m+1} \le 2^{r(m+1)}a_0 + \sum_{k=0}^{m}2^{r(m-k)}b_k
  $$
  holds for all $m \in \N_0$. Consequently, we can apply Minkowski's and Jensen's inequality to deduce that
  \begin{align*}
    \Big(\sum_{m=0}^\infty a_m^p\Big)^{\frac{1}{p}} 
    &\le \Big(\sum_{m=0}^\infty (2^{rm}a_0)^p\Big)^{\frac{1}{p}} + \Big(\sum_{m=0}^\infty \Big(\sum_{k=0}^{m}2^{r(m-k)}b_k\Big)^p \Big)^{\frac{1}{p}} \\
    &\le C + C\Big(\sum_{m=0}^\infty \sum_{k=0}^{m}2^{r(m-k)}b_k^p \Big)^{\frac{1}{p}} \\
    &\le C+C\Big(\sum_{m=0}^\infty b_m^p\Big)^{\frac{1}{p}} 
    \le C + C\|X\|_{W^\alpha_p,(1)} < + \infty. 
  \end{align*} 
  Combining all above estimates, we obtain that
  \begin{align*}
    \Big (\sum_{m=0}^\infty & \Big(2^{m(\alpha-\frac{1}{p})}\Big(\sum_{k=0}^{2^m-1}\|\tilde{X}_{\frac{k}{2^m},\frac{k+1}{2^m}}\|_G^p\Big)^{\frac{1}{p}} \Big)^p \Big)^{\frac{1}{p}}\\
    &\le C \|X\|_{W^\alpha_p,(1)}
    +\sum_{m=0}^\infty \Big(2^{mp(\alpha-\frac{1}{p})}\sum_{k=0}^{2^m-1}\|Y_{\frac{k}{2^m},\frac{k+1}{2^m}}\|_G^p \Big)^{\frac{1}{p}} \\
    &\le C\|X\|_{W^\alpha_p,(1)} + \Big(\sum_{m=0}^\infty a_m^p\Big)^{\frac{1}{p}} 
    \le C + C\|X\|_{W^\alpha_p,(1)} < +\infty,
  \end{align*}
  which gives the bound~\eqref{eq:discrete characterization of Lie group valued Besov functions}. Furthermore, this shows that $\tilde{X}$ can be extended from the dyadic numbers to the whole interval~$[0,1]$ and then the left-hand side of the above inequality will be equal to $\|\tilde{X}\|_{W^\alpha_p,(1)}$. Now, using \cite[Theorem~2.2]{Liu2020} again for $\tilde{X}$ we can conclude that $\tilde{X}$ belongs to $W^\alpha_{p}([0,1];(G, \|\cdot \|_{G}))$ and the condition $\pi_{G,G/K}(\tilde{X}) = X$ is guaranteed by~\eqref{eq:the projection condition}.
\end{proof}

As an application of Lemma~\ref{lem:Generalization of Lemma 13 in Lyons-Victoir}, we obtain the following two results, to which we refer to as \textit{Lyons--Victoir extension theorem for Sobolev paths}. Note that Theorem~\ref{thm:Generalization of Theorem 14 in Lyons-Victoir} and Corollary~\ref{cor:Generalization of Corollary 19 in Lyons-Victoir} are the counterparts of \cite[Theorem~14]{Lyons2007a} and \cite[Corollary~19]{Lyons2007a}, respectively, in the Sobolev setting, and all arguments used for establishing these two results as given in \cite{Lyons2007a} remain valid in the current setting up to changing the H{\"o}lder norms to Sobolev norms.

\begin{theorem}\label{thm:Generalization of Theorem 14 in Lyons-Victoir}
  Let $0<\alpha<1$ and $1\le p\le + \infty$ be such that $\alpha > 1/p$ and $\frac{1}{\alpha}\notin \N\setminus \{1\}$. Let $K$ be a closed normal subgroup of $G^{[\frac{1}{\alpha}]}(\R^d)$. If $X\in W^\alpha_{p}([0,T];G^{[\frac{1}{\alpha}]}(\R^d)/K)$, then there exists a rough path $\X \in W^\alpha_{p}([0,T];G^{[\frac{1}{\alpha}]}(\R^d))$ such that
  $$
    \pi_{G^{[\frac{1}{\alpha}]}(\R^d), G^{[\frac{1}{\alpha}]}(\R^d)/K}(\X) = X.
  $$
\end{theorem}

\begin{corollary}\label{cor:Generalization of Corollary 19 in Lyons-Victoir}
  Let $0<\alpha<1$ and $1\le p\le + \infty$ be such that $\alpha > \frac{1}{p}$ and $\frac{1}{\alpha} \notin \N\setminus \{1\}$. Then, every $\R^d$-valued path $X\in W^\alpha_{p}([0,T];\R^d)$ can be lifted to a weakly geometric Sobolev rough path $\X\in W^{\alpha}_{p}([0,T];G^{[\frac{1}{\alpha}]}(\R^d))$.
\end{corollary}

\begin{remark}
  \begin{enumerate}
    \item In the present paper, we mainly consider the following two closed normal subgroups $K \subset G^{[\frac{1}{\alpha}]}(\R^d)$:
    \begin{itemize}
      \item $K = \exp(W_n)$ for $n = [\frac{1}{\alpha}]$ and $W_n$ defined as in Lemma \ref{lem:Generalization of Lemma 13 in Lyons-Victoir}. Using such a $K$ in Theorem~\ref{thm:Generalization of Theorem 14 in Lyons-Victoir}, we can prove Corollary~\ref{cor:Generalization of Corollary 19 in Lyons-Victoir} above.
      \item For $\alpha \in (\frac{1}{3}, \frac{1}{2})$ and $k, \ell \in \N$ such that $\R^d = \R^k \oplus \R^\ell$. Let $K = \exp([\R^k,\R^\ell])$, where $[\R^k,\R^\ell] = \text{span}\{e_i \otimes e_{k+j} - e_{k+j} \otimes e_i: i \le k, j \le \ell\}$. Using such a $K$ in Theorem~\ref{thm:Generalization of Theorem 14 in Lyons-Victoir} we can prove the existence of a joint lift rough path $\Z \in W^{\alpha}_{p}([0,T];G^{[\frac{1}{\alpha}]}(\R^d))$ over two rough paths $\X \in W^{\alpha}_{p}([0,T];G^{[\frac{1}{\alpha}]}(\R^k))$ and $\Y \in W^{\alpha}_{p}([0,T];G^{[\frac{1}{\alpha}]}(\R^\ell))$, see Proposition~\ref{prop:optimization of 2 levels joint lifts}.
    \end{itemize}
    \item  As in the paper~\cite{Lyons2007a} of Lyons and Victoir, the results in this subsection still hold if one replaces the Euclidean space~$\R^d$ by a general normed space~$V$ (and then equip the Carnot group $G^n(V)$ with a homogeneous norm which can induce a metric), since completeness is actually neither needed for the construction nor for the discrete characterizations.
  \end{enumerate}
\end{remark}

\begin{remark}
  It was shown in \cite{Yang2012} that the Lyons--Victoir extension does not hold true, in general, in the case $1/s \in \N$.
\end{remark}

\section{Optimal rough path extension}\label{sec:optimal lift}

Due to the extension theorems provided in the last section, we know that every path of suitable Sobolev regularity can be lifted to a Sobolev rough path. This leads, in general, to an infinite set of possible rough path lifts for one given $\R^d$-valued path, see for example \cite[Chapter~2.1]{Friz2014}. Therefore, in order to select a specific rough path lift usually stochastic methods such as It\^o or Stratonovich integration are applied. The purpose of this section is to investigate the possibility of using purely deterministic selection criteria, which will be modelled by a given function acting on the set of possible rough path lifts.\smallskip

To be more precise, we assume that $\alpha \in (0,1)$ and $p \in (1, +\infty)$ such that $\alpha > 1/p$ throughout the whole section. Let $K$ be a closed normal subgroup of $G^{N}(\R^d)$.

Given a path $X\in W^{\alpha}_p([0,T];G^{[\frac{1}{\alpha}]}(\R^d)/K)$ we define the set  
\begin{equation*}
  \mathcal{A}(X):= \Big\{ \X \in W^{\alpha}_p([0,T];G^{[\frac{1}{\alpha}]}(\mathbb{R}^d))\,:\, \pi_{G^{[\frac{1}{\alpha}]}(\R^d),G^{[\frac{1}{\alpha}]}(\R^d)/K} (\X) =X\Big\}.
\end{equation*}
The set $\mathcal{A}(X)$ is called \textit{admissible set} of rough path lifts above~$X$. Thanks to the Lyons--Victoir extension theorem for Sobolev paths (Theorem~\ref{thm:Generalization of Theorem 14 in Lyons-Victoir}), the admissible set $\mathcal{A}(X)$ is always non-empty. In words, $\mathcal{A}(X)$ denotes the set of all rough path lifts $\X$ of $X$ which have the same Sobolev regularity as the given path~$X$. 

\begin{problem}\label{prob:optimal extension}
  Given a path $X\in W^{\alpha}_p([0,T];G^{[\frac{1}{\alpha}]}(\R^d)/K)$ and a selection criterion $F\colon \mathcal{A}(X) \to \R \cup \{+\infty\}$, we are looking for an admissible rough path lift $\X^*\in \mathcal{A}(X)$ such that
  \begin{equation}\label{problem:optimal lift}
    F(\X^*)=\min_{\X\in \mathcal{A}(X)} F(\X).
  \end{equation}
\end{problem}

In the sequel, we will demonstrate that, for suitable convex functionals~$F$ and for certain classes of closed subgroups~$K$, the Optimal Extension Problem~\ref{prob:optimal extension} admits indeed a unique solution~$\X^*$. One crucially ingredient will be that the space of Sobolev rough paths (and thus the admissible set~$\mathcal{A}(X)$) can be embedded naturally into a Sobolev space, which is a reflexive Banach space. This property does not hold true if the Sobolev topology is replaced by, e.g., a H\"older or $p$-variation topology, which makes optimization problems with respect to H\"older or $p$-variation topologies very cumbersome.

\subsection{Existence and uniqueness}

A classical approach to solve a minimizing problem like the Optimal Extension Problem~\ref{prob:optimal extension} is to rely on convex analysis and convex optimization. To proceed in our setting via such methods, some care is required. Indeed, since the Lie group $G^{N}(\mathbb{R}^d)$ is not a vector space, there is no canonical notion of convexity on $G^{N}(\mathbb{R}^d)$.

\begin{remark}
  For example, taking $g,h \in G^N(\R^d)$ and $\lambda \in (0,1)$, both elements $\delta_{1 - \lambda}g \otimes \delta_{\lambda}h$ and $g \otimes \delta_{\lambda}(g^{-1} \otimes h)$ can be viewed as convex combination of $g$ and $h$. However, unlike to the vector space case, the relation $\delta_{1 - \lambda}g \otimes \delta_{\lambda}h = g \otimes \delta_{\lambda}(g^{-1} \otimes h)$ fails in general. 
 
  For our minimizing problem~\eqref{problem:optimal lift} over an admissible set $\mathcal{A}(X)$ of rough paths, we expect the notion of convexity on $G^{[\frac{1}{\alpha}]}(\mathbb{R}^d)$ can be inherited by the admissible set $\mathcal{A}(X)$. More precisely, any convex combination of two elements $\mathbf{X}$ and $\mathbf{Y}$ from $\mathcal{A}(X) \subset W^\alpha_p([0,T];G^{[\frac{1}{\alpha}]}(\R^d)) $ should be an element in $\mathcal{A}(X)$. 
\end{remark}

Based on this consideration, one possible choice is to apply the $\log$ mapping to transfer elements $g,h$ of $G^{N}(\R^d)$ into its Lie algebra $\mathcal{G}^N(\R^d)$, then performing the classical convex combinations on the vector space $\mathcal{G}^N(\R^d)$ and finally using the $\exp$ mapping to obtain the convex combinations of $g$ and $h$ in $G^N(\R^d)$:
\begin{equation*}
  \mathbf{C}_{\lambda}(g,h) := \exp\Big((1-\lambda)\log g + \lambda \log h\Big), \quad \lambda \in (0,1).
\end{equation*}
Using this notion of convexity, we ensure that, for any closed normal subgroup $K \subset G^{[\frac{1}{\alpha}]}(\R^d)$, $X \in W^{\alpha}_p([0,T];G^{[\frac{1}{\alpha}]}(\mathbb{R}^d)/K)$, $\mathbf{X},\mathbf{Y} \in \mathcal{A}(X)$ and $\lambda \in (0,1)$, the convex combination of  $\mathbf{X}$ and $\mathbf{Y}$, which is denoted by $\mathbf{Z}^\lambda := \mathbf{C}_{\lambda}(\mathbf{X},\mathbf{Y})$, satisfies $\pi_{G^{[\frac{1}{\alpha}]}(\R^d),G^{[\frac{1}{\alpha}]}(\R^d)/K} (\Z^\lambda) =X$. However, in general we may not have $\Z^\lambda \in W^{\alpha}_p([0,T];G^{[\frac{1}{\alpha}]}(\mathbb{R}^d))$ because of the mixed Lie brackets of $\log \X$ and $\log \Y$ from the Campbell--Baker--Hausdorff formula. In order to annihilate the effect caused by these mixed Lie brackets, we need the Lie algebra $\mathcal{K}$ of $K$ to be a subspace of ``the highest layer" of the Lie algebra~$\mathcal{G}^{[\frac{1}{\alpha}]}(\R^d)$.

\begin{lemma}\label{lem:admissible set is convex}
  Let $K$ be a closed subgroup of $\exp(W_{[\frac{1}{\alpha}]})$ and $X \in W^{\alpha}_p([0,T];G^{[\frac{1}{\alpha}]}(\R^d)/K)$. If $\mathbf{X}, \mathbf{Y}\in \mathcal{A}(X)$, then $\mathbf{Z}^\lambda  = \mathbf{C}_{\lambda}(\mathbf{X},\mathbf{Y}) \in \mathcal{A}(X)$ for every $\lambda \in (0,1)$. 
  
  Moreover, by viewing $G^{[\frac{1}{\alpha}]}(\R^d)$ as a subset of the affine vector space $\bigoplus_{i=0}^{[\frac{1}{\alpha}]}  (\R^d)^{\otimes i}$, it holds that $\Z^\lambda_{s,t} = \X_{s,t} + \lambda(\Y_{s,t} - \X_{s,t})$ for every $\lambda \in (0,1)$ and $s,t$ in $[0,T]$.
\end{lemma}

\begin{proof}
  Fix a $\lambda \in (0,1)$, we write for simplicity $\mathbf{Z}$ instead of $\mathbf{Z}^\lambda$ during the proof. From the definition of $\Z$ it suffices to show that $\Z$ belongs to $W^\alpha_p([0,T];G^{[\frac{1}{\alpha}]}(\R^d))$.
  
  Let $g_t := \log \X_t$ and $h_t := \log \Y_t$ for $t \in [0,T]$. We write $g_t = g^{1}_t + \dots + g^{[\frac{1}{\alpha}]}_t$ and $h_t = h^{1}_t + \dots + h^{[\frac{1}{\alpha}]}_t$ with $g^i_t$ and $h^i_t$ belonging to $W_i$ for $i = 1, \dots , [\frac{1}{\alpha}]$. Since $K \subset \exp(W_{[\frac{1}{\alpha}]})$, $X$ takes values in the quotient group $G^{[\frac{1}{\alpha}]}(\R^d)/K$ and $\X, \Y$ are elements in $\mathcal{A}(X)$, it holds that $g^i_t = h^i_t$ for $i = 1, \dots, [\frac{1}{\alpha}] - 1$. By definition, we have
  $$
    \Z_t = \exp\Big((1-\lambda)g_t + \lambda h_t\Big) = \exp\Big(g^1_t + \dots + g^{[\frac{1}{\alpha}]-1}_t + (1- \lambda)g^{[\frac{1}{\alpha}]}_t + \lambda h^{[\frac{1}{\alpha}]}_t\Big).
  $$
  As a consequence, by Campbell--Baker--Hausdorff formula, we obtain that for all $s<t$ in $[0,T]$, 
  \begin{align*}
    \Z_{s,t} = \Z_s^{-1} \otimes \Z_t = \exp\Big(g^1_{s,t} + \dots + g^{[\frac{1}{\alpha}]-1}_{s,t} + (1- \lambda)g^{[\frac{1}{\alpha}]}_{s,t} + \lambda h^{[\frac{1}{\alpha}]}_{s,t} + M_{s,t} \Big),
  \end{align*}
  where $M_{s,t}$ contains all Lie brackets involving $g^i_s$ and $g^i_t$ for $i = 1, \dots, [\frac{1}{\alpha}] - 1$. On the other hand, again by Campbell--Baker--Hausdorff formula we get
  $$
    \X_{s,t} = \exp\Big(g^1_{s,t} + \dots + g^{[\frac{1}{\alpha}]-1}_{s,t} + g^{[\frac{1}{\alpha}]}_{s,t} + M_{s,t} \Big)
  $$
  and 
  $$
    \Y_{s,t} = \exp\Big(g^1_{s,t} + \dots + g^{[\frac{1}{\alpha}]-1}_{s,t} + h^{[\frac{1}{\alpha}]}_{s,t} + M_{s,t} \Big).
  $$
  Hence, the above calculations reveal that $\log \Z_{s,t} = (1 - \lambda) \log \X_{s,t} + \lambda \log \Y_{s,t}$. Now, using the equivalence of homogeneous norms on $G^{[\frac{1}{\alpha}]}(\R^d)$ (see \cite[Proposition~7]{Lyons2007a}) we observe that
  \begin{align*}
    d_{cc}(\Z_s,\Z_t) = \|\Z_{s,t}\|&\sim \max_{i=1,\dots,[\frac{1}{\alpha}]} |(\log \Z_{s,t})^i|^{\frac{1}{i}} 
    =\max_{i=1,\dots,[\frac{1}{\alpha}]} |(1-\lambda)(\log \X_{s,t})^i + \lambda (\log \Y_{s,t})^i|^{\frac{1}{i}}\\
    &\le (1 - \lambda)^{\alpha}\max_{i=1,\dots,[\frac{1}{\alpha}]} |(\log \X_{s,t})^i|^{\frac{1}{i}} + \lambda^\alpha \max_{i=1,\dots,[\frac{1}{\alpha}]} |(\log \Y_{s,t})^i|^{\frac{1}{i}} \\
    &\lesssim (1 - \lambda)^{\alpha}d_{cc}(\X_s,\X_t) + \lambda^\alpha d_{cc}(\Y_s,\Y_t),
  \end{align*}
  where the proportional constant is independent of $s$ and $t$. Since $\X$ and $\Y$ are elements in $W^\alpha_p([0,T]; G^{[\frac{1}{\alpha}]}(\R^d))$, the above estimates imply that $\Z \in W^\alpha_p([0,T]; G^{[\frac{1}{\alpha}]}(\R^d))$.
  
  The above observations show that
  \begin{align*}
    \Z_{s,t}
    &=1+ \sum_{i =1}^{[\frac{1}{\alpha}]-1}\pi_i(\Y_{s,t}) + (1-\lambda)\pi_{[\frac{1}{\alpha}]}(\X_{s,t}) + \lambda \pi_{[\frac{1}{\alpha}]}(\Y_{s,t}),\quad s,t \in [0,T],
  \end{align*}
  which ensures that $\Z_{s,t} = \X_{s,t} + \lambda(\Y_{s,t} - \X_{s,t})$, and completes the proof.
\end{proof}

Thanks to Lemma~\ref{lem:admissible set is convex}, if $K$ is a closed subgroup of $\exp(W_{[\frac{1}{\alpha}]})$, then for any Sobolev path $X \in  W^{\alpha}_p([0,T];G^{[\frac{1}{\alpha}]}(\R^d)/K)$, the admissible set $\mathcal{A}(X)$ is convex. In particular, the convexity on the Lie group $G^{[\frac{1}{\alpha}]}(\R^d)$ coincides with the classical convexity on vector spaces when we embed $G^{[\frac{1}{\alpha}]}(\R^d)$ into the affine vector space $\bigoplus_{i=0}^{[\frac{1}{\alpha}]}(\R^d)^{\otimes i}$. This will allow us to rely on convex analysis on vector spaces to prove the existence of a (unique) solution to the Optimal Extension Problem~\ref{prob:optimal extension}.

\begin{theorem}\label{thm:existence and uniqueness of optimal extension}
  Let $K$ be a closed subgroup of $\exp(W_{[\frac{1}{\alpha}]})$, $X \in  W^{\alpha}_p([0,T];G^{[\frac{1}{\alpha}]}(\R^d)/K)$ and $F \colon \mathcal{A}(X) \rightarrow \R \cup \{+\infty\}$ be a proper functional defined on $\mathcal{A}(X)$, that is, $F(\X) \in \R$ holds for at least one $\X \in \mathcal{A}(X)$. If $F$ can be extended to a coercive, convex and lower semi-continuous functional $\overline{F}$ defined on the Banach space $W^\alpha_p([0,T]; \bigoplus_{i=1}^{[\frac{1}{\alpha}]} (\R^d)^{\otimes i})$ such that $\overline{F} = +\infty$ outside $\mathcal{A}(X)$, then the Optimal Extension Problem~\ref{prob:optimal extension} admits a solution~$\X^*$ in $\mathcal{A}(X)$. If in addition $\overline{F}$ is strictly convex on $\mathcal{A}(X)$, then the solution~$\X^*$ to \eqref{problem:optimal lift} is unique.
\end{theorem}

\begin{proof}
  With a slight abuse of notation, let $| \cdot |_{W^\alpha_p}$ denote the Sobolev norm on the reflexive Banach space $W^\alpha_p([0,T]; \bigoplus_{i=1}^{[\frac{1}{\alpha}]} (\R^d)^{\otimes i})$. By the ball-box estimate (cf. \cite[Proposition~7.49]{Friz2010}), for any $\X \in W^{\alpha}_p([0,T];G^{[\frac{1}{\alpha}]}(\R^d))$ it holds that $|\X|_{W^\alpha_p} \lesssim \| \X \|_{W^\alpha_p}$, where the proportional constant only depends on $\sup_{t \in [0,T]}\|\X_t\|_{cc}$. Hence, by Lemma~\ref{lem:admissible set is convex} we can embed $\mathcal{A}(X)$ as a convex subset into the Banach space $W^\alpha_p([0,T]; \bigoplus_{i=1}^{[\frac{1}{\alpha}]} (\R^d)^{\otimes i})$, where the convexity is induced by the usual addition on $\bigoplus_{i=1}^{[\frac{1}{\alpha}]} (\R^d)^{\otimes i}$. Hence, for any $F$ satisfies all properties of Theorem~\ref{thm:existence and uniqueness of optimal extension}, the assertion follows immediately from the standard results in convex analysis on reflexive Banach spaces, see e.g. \cite[Theorem~2.5.1]{Zualinescu2002}.
\end{proof}

\subsection{Example: minimal Sobolev extension}

As an exemplary choice of the selection criterion~$F$, we shall take $F$ to be the inhomogeneous Sobolev norm. In this case, the optimal extension problem~\ref{problem:optimal lift} asks to find the (unique) rough path lift~$\X^*$ of a given path~$X$ which has the minimal inhomogeneous Sobolev norm or, in other words, is of ``minimal length'' among all possible rough path lifts above~$X$.

Recall that for $\X \in W^{\alpha}_p([0,T];G^{[\frac{1}{\alpha}]}(\R^d))$ the \textit{inhomogeneous Sobolev norm} is defined by
\begin{equation}\label{eq:inhomogeneous Sobolev norm}
  \vertiii \X_{W^{\alpha}_p} := \sum_{i =1}^{[\frac{1}{\alpha}]}\Big(\iint_{[0,T]^2} \frac{|\pi_i(\X_{s,t})|^{\frac{p}{i}}}{|t-s|^{\alpha p + 1}} \dd s \dd t\Big)^{\frac{i}{p}},
\end{equation}
cf. \cite{Liu2021b}.

In the following, we shall show that there exists indeed a unique minimizer~$\X^*$ to the optimal extension problem~\ref{problem:optimal lift} with respect to the inhomogeneous Sobolev norm $F (\cdot) := \vertiii \cdot_{W^{\alpha}_p}$. In order to rely on Theorem~\ref{thm:existence and uniqueness of optimal extension}, we shall verify that its assumptions are fulfilled by $F (\cdot) = \vertiii \cdot_{W^{\alpha}_p}$. As a first step, we prove that~$F$ is a strictly convex functional on the admissible set.


\begin{lemma}\label{lem:convexity of the inhomogeneous norm}
  Let $K$ be a closed subgroup of $\exp(W_{[\frac{1}{\alpha}]})$ and $X \in W^{\alpha}_p([0,T];G^{[\frac{1}{\alpha}]}(\R^d)/K)$.
  The functional $F (\cdot)=\vertiii \cdot_{W^{\alpha}_p}\colon \mathcal{A}(X)\to \mathbb{R} $ is strictly convex.
\end{lemma}

\begin{proof}
  Fix a $\lambda \in (0,1)$ and $\X, \Y \in \mathcal{A}(X)$, we again write $\mathbf{Z}$ for $\mathbf{Z}^\lambda = \mathbf{C}_\lambda(\X,\Y)$.  By Lemma~\ref{lem:admissible set is convex}, we have $\Z \in \mathcal{A}(X)$ and $\Z_{s,t} = (1-\lambda)\X_{s,t} + \lambda \Y_{s,t}$ for all $s,t \in [0,T]$. Hence, we obtain that for each $i = 1, \dots, [\frac{1}{\alpha}]$,
  \begin{align*}
    \bigg(\iint_{[0,T]^2} \frac{|\pi_i(\Z_{s,t})|^{\frac{p}{i}}}{|t-s|^{\alpha p + 1}}\dd s \dd t \bigg)^{\frac{i}{p}} 
    =  \bigg(\iint_{[0,T]^2} \frac{|(1-\lambda)\pi_i(\X_{s,t}) + \lambda \pi_i(\Y_{s,t})|^{\frac{p}{i}}}{|t-s|^{\alpha p + 1}}\dd s \dd t \bigg)^{\frac{i}{p}}. 
  \end{align*}
  By the uniform convexity of the $L^{\frac{p}{i}}([0,T]^2, \d s \dd t)$-norm (note that $\frac{p}{i} > 1$ for all $i = 1, \dots, [\frac{1}{\alpha}]$) we obtain that
  \begin{align*}
    & \bigg(\iint_{[0,T]^2} \frac{|\pi_i(\Z_{s,t})|^{\frac{p}{i}}}{|t-s|^{\alpha p + 1}}\dd s \dd t \bigg)^{\frac{i}{p}} \\
    &\qquad\quad\le (1-\lambda)  \bigg(\iint_{[0,T]^2} \frac{|\pi_i(\X_{s,t})|^{\frac{p}{i}}}{|t-s|^{\alpha p + 1}}\dd s \dd t \bigg)^{\frac{i}{p}} + \lambda   \bigg(\iint_{[0,T]^2} \frac{|\pi_i(\Y_{s,t})|^{\frac{p}{i}}}{|t-s|^{\alpha p + 1}}\dd s \dd  t\bigg)^{\frac{i}{p}}
  \end{align*}
  and the equality holds only when $\pi_i(\X_{s,t}) = \pi_i(\Y_{s,t})$ for all $s,t \in [0,T]$ since $\X$ and $\Y$ are continuous paths. Now in view of the definition of the inhomogeneous Sobolev norm (cf. \eqref{eq:inhomogeneous Sobolev norm}) we can conclude that $\vertiii \cdot_{W^{\alpha}_p}$ is convex on $\mathcal{A}(X)$. Finally, if $\vertiii \Z_{W^{\alpha}_p} = (1-\lambda)\vertiii \X_{W^{\alpha}_p} + \lambda \vertiii \Y_{W^{\alpha}_p}$ holds, then from the above observations we must have $\pi_i(\X_{s,t}) = \pi_i(\Y_{s,t})$ for all $s,t \in [0,T]$ and for all $i = 1, \dots, [\frac{1}{\alpha}]$, which ensures that $\X = \Y$. This gives the strict convexity of~$\vertiii \cdot_{W^{\alpha}_p}$.
\end{proof}

As a next step, we extend the functional $F(\cdot):=\vertiii \cdot_{W^{\alpha}_p}$ to a functional $\overline{F}$, defined on the whole Sobolev rough path space, in the following way:
\begin{equation}\label{eq:defnition of the cost functional}
  \overline{F}\colon W^{\alpha}_p\Big([0,T]; \bigoplus_{i=1}^{[\frac{1}{\alpha}]} (\R^d)^{\otimes i}\Big) \to \R \cup \{ +\infty \}\quad \text{with}\quad
  \overline{F}(\mathbf{X}) := \begin{cases}
  \vertiii \X_{W^{\alpha}_p},& \text{if } \mathbf{X} \in \mathcal{A}(X), \\
  +\infty,& \text{otherwise}.
  \end{cases}
\end{equation}
As we establish in the next lemma, the functional $\overline{F}$ fulfills the assumptions of Theorem~\ref{thm:existence and uniqueness of optimal extension}.

\begin{lemma}\label{lem:properties of cost functional}
  Let $K$ be a closed subgroup of $\exp(W_{[\frac{1}{\alpha}]})$ and $X \in W^{\alpha}_p([0,T];G^{[\frac{1}{\alpha}]}(\R^d)/K)$.
  The functional $\overline{F}$ defined as in~\eqref{eq:defnition of the cost functional} is strictly convex, coercive and lower semi-continuous.
\end{lemma}

\begin{proof}
  (1) $\overline{F}$ is strictly convex: By construction of $\overline{F}$, it suffices to show that $\vertiii \cdot_{W^{\alpha}_p}$ is strictly convex on $\mathcal{A}(X)$, which is the content of Lemma~\ref{lem:convexity of the inhomogeneous norm}. \smallskip
 
  (2) $\overline{F}$ is coercive: We show that for any $a \in \R$, the set $\{\overline{F} \leq a \}\subset W^{\alpha}_p([0,T]; \bigoplus_{i=1}^{[\frac{1}{\alpha}]} (\R^d)^{\otimes i})$ is bounded. In view of the definition of $\overline{F}$ it is enough to check that 
  $$
    \big\{\mathbf{X} \in \mathcal{A}(X) : \vertiii \X_{W^{\alpha}_p}\leq a\big \}
  $$
  is bounded w.r.t. the $|\cdot|_{W^{\alpha}_p}$-norm on $W^{\alpha}_p([0,T];\bigoplus_{i=1}^{[\frac{1}{\alpha}]} (\R^d)^{\otimes i})$, where $| \cdot |_{W^\alpha_p}$ denotes again the Sobolev norm $W^\alpha_p([0,T]; \bigoplus_{i=1}^{[\frac{1}{\alpha}]} (\R^d)^{\otimes i})$. By the equivalence of homogeneous norms (cf. \cite[Theorem~7.44]{Friz2010}), for any $s,t$ in $[0,T]$ we have
  $$
    d_{cc}(\X_s,\X_t) \sim  \max_{i=1,\dots,[\frac{1}{\alpha}]}|\pi_i(\X_{s,t})|^{\frac{1}{i}},
  $$
  which implies that 
  \begin{equation}\label{eq:bound by inhomogeneous Sobolev norm}
    \|\X\|_{W^{\alpha}_p} \lesssim_{a} \vertiii \X_{W^{\alpha}_p}^{\alpha}.
  \end{equation}
  Furthermore, by Sobolev embedding (see \cite[Theorem~2]{Friz2006}) we have $\|\X\|_{\frac{1}{\alpha}\var} \lesssim \|\X\|_{W^{\alpha}_p}$, and therefore $\sup_{t \in [0,T]}\|\X_{0,t}\|_{cc} \le C$ for some constant $C$ only depends on $a$. From the proof of Theorem~\ref{thm:existence and uniqueness of optimal extension} we know that $|\X|_{W^{\alpha}_p} \lesssim \|\X\|_{W^{\alpha}_p}$. Hence, we can conclude that $|\X|_{W^{\alpha}_p} \le f(a)$ for some continuous increasing function $f$ uniformly over all $\X$ satisfying $\vertiii \X_{W^{\alpha}_p} \le a$.\smallskip
  
  (3) $\overline{F}$ is lower semi-continuous. Suppose that $(\mathbf{X}^n)_{n \ge 1}$ is a sequence converging to $\mathbf{X}$ in $ W^{\alpha}_p([0,T]; \bigoplus_{i=1}^{[\frac{1}{\alpha}]} (\R^d)^{\otimes i})$, and  without loss of generality we may assume that every $\mathbf{X}^n$ is in $\mathcal{A}(X)$. Since $\alpha > \frac{1}{p}$, Sobolev embedding (see \cite[Theorem~2]{Friz2006}) provides that $\X^n$ converges to $\X$ also w.r.t. the $\|\cdot\|_{\frac{1}{\alpha}\var}$-norm, which in turn implies that $\mathbf{X}^n_{s,t}$ converges to $\mathbf{X}_{s,t}$ for every $s,t \in [0,T]$. By Fatou's lemma, we can deduce that $\vertiii \X_{W^{\alpha}_p} \le \liminf_{n \rightarrow \infty}\vertiii {\X^n}_{W^{\alpha}_p}$. If $\liminf_{n \rightarrow \infty} \vertiii{\X^n}_{W^{\alpha}_p} = + \infty$, then we always have $\overline{F}(\mathbf{X}) \leq \liminf_{n \rightarrow \infty} \overline{F}(\mathbf{X}^n)$ (notice that in this case we do not need $\mathbf{X} \in \mathcal{A}(X)$). If $\liminf_{n \rightarrow \infty} \vertiii{\X^n}_{W^{\alpha}_p} < +\infty$, then $\mathbf{X}$ has finite $\vertiii \cdot_{W^{\alpha}_p}$-norm and the pointwise convergence ensures that $\mathbf{X}_t \in G^{[\frac{1}{\alpha}]}(\R^d)$  and $\pi_{G^{[\frac{1}{\alpha}]}(\R^d),G^{[\frac{1}{\alpha}]}(\R^d)/K}(\mathbf{X}_t) = X_t$ hold for all $t \in [0,T]$. These arguments together with the Bound~\eqref{eq:bound by inhomogeneous Sobolev norm} imply that $\mathbf{X} \in \mathcal{A}(X)$. Hence, in this case we have $\overline{F}(\mathbf{X}) \leq \liminf_{n \rightarrow \infty} \overline{F}(\mathbf{X}^n)$ as well.
\end{proof}

Based on Lemma~\ref{lem:convexity of the inhomogeneous norm} and \ref{lem:properties of cost functional} and as an application of Theorem~\ref{thm:existence and uniqueness of optimal extension}, we obtain the existence of a unique rough path lift~$\X^*$ of a given path~$X$, which is of ``minimal length''.

\begin{proposition}\label{prop:optimization of inhomogeneous Sobolev norm}
  Let $K$ be a closed subgroup of $\exp(W_{[\frac{1}{\alpha}]})$ and $X \in  W^{\alpha}_p([0,T];G^{[\frac{1}{\alpha}]}(\R^d)/K)$. Then, there exists a unique rough path lift $\X^* \in \mathcal{A}(X)$ of minimal inhomogeneous Sobolev norm $\vertiii \cdot_{W^{\alpha}_p}$, i.e.,
  $$  
    \vertiii {\X^*}_{W^{\alpha}_p} \le \vertiii{\X}_{W^{\alpha}_p} \quad \text{ for all}\quad \X \in \mathcal{A}(X).
  $$
  In particular, given a path $X \in  W^{\alpha}_p([0,T];\R^d)$ for $\alpha\in (1/3,1/2)$, there exists a unique rough path lift $\X^* \in \mathcal{A}(X)$ such that
  $$  
    \vertiii {\X^*}_{W^{\alpha}_p} \le \vertiii{\X}_{W^{\alpha}_p} \quad \text{for all}\quad \X \in \mathcal{A}(X).
  $$
\end{proposition}

For simplicity let us now take that $T=1$. In Section~\ref{subsec:Sobolev rough path} we introduced an equivalent Sobolev norm $\| \cdot \|_{W^\alpha_p,(1)}$ on the Sobolev space $W^\alpha_p([0,1];G^{[\frac{1}{\alpha}]}(\R^d))$, see formula~\eqref{eq:discrete Sobolev norm}. Correspondingly, as in \cite[(2.4)]{Liu2021}, one can define its inhomogeneous counterpart on $W^\alpha_p([0,1];G^{[\frac{1}{\alpha}]}(\R^d))$, which will be denoted by $\vertiii{\cdot}_{W^\alpha_p,(1)}$ such that 
\begin{equation}\label{eq: inhomogeneous discrete Sobolev norm}
  \vertiii{\X}_{W^\alpha_p,(1)} := \sum_{k=1}^{[\frac{1}{\alpha}]} \vertiii{\X}_{W^\alpha_p,(1);k},
  \quad \text{for}\quad \X \in W^\alpha_p([0,1];G^{[\frac{1}{\alpha}]}(\R^d)),
\end{equation}
where 
\begin{equation*}
  \vertiii{\X}_{W^\alpha_p,(1);k} := \Big(\sum_{j \ge 0}2^{j(\alpha p -1)}\sum_{i = 1}^{2^j}|\pi_k(\X_{ \frac{i-1}{2^j}T, \frac{i}{2^k}T})|^{\frac{p}{k}}\Big)^{\frac{k}{p}}.
\end{equation*}

Let $F(\cdot) := \vertiii{\cdot}_{W^\alpha_p,(1)}$, we will show that such $F$ is a strictly convex functional on the admissible set as the classical inhomogeneous Sobolev norm $\vertiii{\cdot}_{W^\alpha_p}$ introduced before.

\begin{lemma}\label{lem:convexity of the inhomogeneous discrete norm}
  Let $K$ be a closed subgroup of $\exp(W_{[\frac{1}{\alpha}]})$ and $X \in W^{\alpha}_p([0,1];G^{[\frac{1}{\alpha}]}(\R^d)/K)$. The functional $F (\cdot)=\vertiii \cdot_{W^{\alpha}_p,(1)}\colon \mathcal{A}(X)\to \mathbb{R} $ is strictly convex.
\end{lemma}

\begin{proof}
  First let us introduce some notations. For $\sigma \in \R$, $q \in (0,\infty]$ and a Banach space~$U$, the weighted $\ell^q$-space $\ell^\sigma_q(U)$ is defined by 
  $$
    \ell^\sigma_q(U) := \Big\{\xi = (\xi_j)_{j\ge 0}: \xi_j \in U, \|\xi\|_{\ell^\sigma_p(U)} := \Big(\sum_{j\ge 0}(2^{\sigma j}\|\xi_j\|)^q\Big)^{\frac{1}{q}} <+ \infty\Big\}.
  $$
  Fix a $k \in \{1,\ldots,[\frac{1}{\alpha}]\}$ and an $\X \in  W^\alpha_p([0,1];G^{[\frac{1}{\alpha}]}(\R^d))$, we define $\sigma_k:= k(\alpha - 1/p)$, $q_k:= p/k$ and $U_k := L^{q_k}(D,\mu;(\R^d)^{\otimes k})$, where $D$ are the dyadic numbers on $[0,1]$ and $\mu$ is the counting measure on $D$. Furthermore, let $\xi(\X)^k := (\xi(\X)^k_j)_{j \ge 0}$ such that for $a \in D$, $j \ge 1$,
  \begin{align*}
    \xi(\X)^k_j(a) := 
    \begin{cases}
      \pi_k(\X_{ \frac{i-1}{2^j}, \frac{i}{2^k}}),& \text{if } a = \frac{i}{2^j} \text{ for } i = 0,\ldots,2^j - 1,\\
       0,& \text{otherwise},
    \end{cases}
  \end{align*}
  and for $j = 0$,
  \begin{align*}
    \xi(\X)_0(a) := 
    \begin{cases}
      \pi_k(\X_{0,1}),& \text{for } a = 1,\\
      0,& \text{otherwise}.  
    \end{cases}
  \end{align*}
  One can verify that $\xi(\X)^k$ belongs to the space $\ell^{\sigma_k}_{q_k}(U_k)$ and $\|\xi(\X)^k\|_{\ell^{\sigma_k}_{q^k}(U_k)} = \vertiii{\X}_{W^\alpha_p,(1);k}$. Since $\ell^{q_k}$ and $L^{q_k}$-spaces are all uniformly convex when $q_k \in (1,\infty)$, we obtain the strict convexity of the norm $\| \cdot\|_{\ell^{\sigma_k}_{q_k}(U_k)}$ for each $k$. As a consequence, our claim follows by noting that for any $\lambda \in (0,1)$, $\X,\Y \in \mathcal{A}(X)$, its convex combination $\Z := (1-\lambda)\X + \lambda \Y$ is again an element in $\mathcal{A}(X)$ and for each $k$ it holds that $\xi(\Z)^k = (1-\lambda)\xi(\X)^k + \lambda \xi(\Y)^k$, see the proof of Lemma~\ref{lem:convexity of the inhomogeneous norm}.
\end{proof}

Now we extend $F(\cdot) = \vertiii{\cdot}_{W^\alpha_p,(1)}$ to a functional $\overline{F}$ defined on the whole Sobolev space by following the same manner as in~\eqref{eq:defnition of the cost functional}. We will denote by $|\cdot|_{W^\alpha_p,(1)}$ the counterpart of the Sobolev norm~\eqref{eq:discrete Sobolev norm} on the Banach space $W^\alpha_p([0,1];\bigoplus_{i=1}^{[\frac{1}{\alpha}]} (\R^d)^{\otimes i})$ (i.e., by replacing the metric $d_{cc}$ through the Euclidean distance on $\bigoplus_{i=1}^{[\frac{1}{\alpha}]} (\R^d)^{\otimes i}$). By \cite[Theorem~2.2]{Liu2020} we know that $|\cdot|_{W^\alpha_p,(1)}$ is equivalent to the Sobolev norm $|\cdot|_{W^\alpha_p}$. Hence, all arguments for establishing Lemma~\ref{lem:properties of cost functional} remain valid for the current functional $\overline{F}$ and we can deduce that $\overline{F}$ induced by $\vertiii{\cdot}_{W^\alpha_p,(1)}$ is coercive and lower semi-continuous. So, together with Lemma~\ref{lem:convexity of the inhomogeneous discrete norm} we recover all results obtained for the previously defined inhomogeneous Sobolev norm, which will be summarized in the corollary below:

\begin{remark}
  Proposition~\ref{prop:optimization of inhomogeneous Sobolev norm} is also true if one replaces $\vertiii{\cdot}_{W^\alpha_p}$ by $\vertiii{\cdot}_{W^\alpha_p,(1)}$ for $T=1$ and with the obvious modifications if $T\neq 1$.
\end{remark}

An explicit characterization of the optimal rough path lift~$\X^*$ given the inhomogeneous Sobolev norm as selection criterion would certainly be very interesting. However, this seems to be a possibly very hard question outside the scope of the current article. Therefore, we leave it at this point for future research.

\subsection{Example: Stratonovich lift of a Brownian motion}

Analogously to the inhomogeneous Sobolev norm $ \vertiii{\cdot}_{W^\alpha_p,(1)}$ (see \eqref{eq: inhomogeneous discrete Sobolev norm}), one can introduce a discrete inhomogeneous Sobolev distance on $W^\alpha_p([0,1];G^{[\frac{1}{\alpha}]})$ given by
\begin{equation*}
  \hat{\rho}_{W^\alpha_p} (\X^1,\X^2):= \sum_{k =1}^{[\frac{1}{\alpha}]} \hat{\rho}^{(k)}_{W^\alpha_p}(\X^1,\X^2),
  \quad \text{for}\quad \X^1,\X^2 \in W^\alpha_p([0,1];G^{[\frac{1}{\alpha}]}(\R^d)),
\end{equation*}
where 
\begin{equation*}
  \hat{\rho}^{(k)}_{W^\alpha_p}(\X^1,\X^2) := \Big(\sum_{j \ge 0} 2^{j(\alpha p -1)}\sum_{i=1}^{2^j} |\pi_k(\X^1_{(i-1)2^{-j},i2^{-j}} -   \X^2_{(i-1)2^{-j},i2^{-j}})|^{\frac{p}{k}}\Big)^{\frac{k}{p}}.
\end{equation*}

\begin{remark}
  The discrete inhomogeneous Sobolev distance $\hat{\rho}_{W^\alpha_p} $ plays a crucial role in establishing the local Lipschitz continuity of the It{\^o}--Lyons map on the space of Sobolev rough paths, see \cite{Liu2021}.
\end{remark}

The aim of this subsection is to consider the sample paths of a $d$-dimensional Brownian motion $(B_t)_{t\in[0,T]}$ on a probability space $(\Omega,\mathcal{F},\mathbb{P})$. It is well-known that the sample paths of a Brownian motion have Sobolev regularity with $\alpha \in (1/3,1/2)$ and $p\in (1,+\infty)$ almost surely, see e.g. \cite{Rosenbaum2009}. Hence, we focus in this subsection on the case that $\alpha \in (1/3,1/2)$, $p\in (1,+\infty)$, $\alpha > 1/p$ and take again for simplicity $T = 1$.\smallskip 

Given a Sobolev path $X \in W^\alpha_p([0,1];\R^d)$ we denote by $X^n$ the dyadic interpolation of $X$ based on the $n$-th dyadic points on $[0,1]$, that is, for $t \in [\frac{k-1}{2^n},\frac{k}{2^n}]$, 
\begin{equation*}
  X^n_t := X_{\frac{k-1}{2^j}} + 2^n\bigg(t - \frac{k-1}{2^j}\bigg)\Delta^n_kX,
  \quad \text{where}\quad
  \Delta^n_kX := X_{\frac{k}{2^n}} - X_{\frac{k-1}{2^n}}.
\end{equation*}
Furthermore, let $S(X^n)$ be the canonical lift of $X^n$ into $G^2(\R^d)$, namely 
\begin{equation*}
  S(X^n) := (1,X^n, \mathbb{X}^n)\quad \text{with}\quad \mathbb{X}^{n}_{s,t} := \int_s^t X^n_{s,r} \otimes \d X^n_{r}.
\end{equation*}
A direct calculation shows that (or see \cite{Ledoux2002}) for all $n \ge m$,
\begin{equation*}
  \mathbb{X}^m_{\frac{k-1}{2^n}, \frac{k}{2^n}} = \frac{1}{2}(2^{m-n})^2 \Delta^m_lX \otimes \Delta^m_lX,
\end{equation*}
where $l$ is the unique integer such that 
\begin{equation*}
  \frac{l-1}{2^m} \le \frac{k-1}{2^n} < \frac{k}{2^n} \le \frac{l}{2^m};
\end{equation*}
and for $n \le m$, it holds that
\begin{equation*}
  \mathbb{X}^m_{\frac{k-1}{2^n}, \frac{k}{2^n}} = \frac{1}{2}\Delta^n_kX \otimes \Delta^n_kX + \frac{1}{2}\sum_{r<l, r,l=2^{m-n}(k-1)+1}^{2^{m-n}k}\Big(\Delta^m_rX \otimes \Delta^m_lX - \Delta^m_lX \otimes \Delta^m_rX\Big).
\end{equation*}
For $m\in \N$, we define
\begin{equation}\label{eq:Stratonovich approximation functional}
  F^m(\X) := \hat{\rho}_{W^\alpha_p} (\X,S(X^m)), \quad \X\in W^\alpha_p([0,1];G^{2}(\R^d)),
\end{equation}
and 
\begin{equation*}\label{eq:Stratonovich functional}
  F(\X):= \limsup_{n\to \infty} F^m(\X).
\end{equation*}
In the following we consider the optimization problem
\begin{equation}\label{eq:Stratonovich is optimal}
  \min_{\X\in \mathcal{A}(X)} F(\X),
\end{equation}
where $F$ is defined as in \eqref{eq:Stratonovich functional} and we recall that $\mathcal{A}(X)$ denotes all possible rough path lifts above the given Sobolev path $X\in W^\alpha_p([0,1];\R^d)$.\smallskip

Coming back to the Brownian motion $(B_t)_{t\in[0,1]}$, we recall that it can be lifted almost surely to a rough path $\mathbf{B}^s=(1,B,\mathbb{B}^{\text{s}})$ via Stratonovich integration, see e.g. \cite[Chapter~13]{Friz2010}. It turns out that the rough path $\mathbf{B}^s=(1,B,\mathbb{B}^{\text{s}})$ is the unique optimal solution to the optimization problem~\eqref{eq:Stratonovich is optimal} almost surely.

\begin{theorem}\label{thm:Stratonovich is optimal}
  For almost all $\omega\in \Omega$, the Stratonovich rough path lift 
  \begin{equation*}
    \mathbf{B}^s(\omega)=(1,B(\omega),\mathbb{B}^{\text{s}}(\omega))\in W^\alpha_p([0,1];G^2(\R^d))
  \end{equation*}
  is the unique minimizer of the optimization problem~\eqref{eq:Stratonovich is optimal} given $F$ as in \eqref{eq:Stratonovich functional} and $\mathcal{A}(B(\omega))$.
\end{theorem}

Before proving Theorem~\ref{thm:Stratonovich is optimal}, we need to do some preliminary work. Using the piecewise linear approximations along the dyadic points, we obtain the following lemma.

\begin{lemma}\label{lemma: optimizer for each dyadic interpolcation}
  Let $\alpha \in (1/3,1/2)$ and $p \in (1,+\infty)$ be such that $\alpha > 1/p$. Let $X \in W^\alpha_p([0,1];\R^d)$ be fixed and let $X^n$ be defined as above. For each $m\in \mathbb{N}$ one has:
  \begin{enumerate}
    \item $S(X^m)$ belongs to $W^\alpha_p([0,1];G^2(\R^d))$. 
    \item Given $X$ and the functional $F^m$ as defined in \eqref{eq:Stratonovich approximation functional}, the Optimal Extension Problem~\ref{prob:optimal extension} admits a unique minimizer in $\mathcal{A}(X)$.
  \end{enumerate}
\end{lemma}

\begin{proof}
  (1) Since every $X^m$ is Lipschitz continuous and therefore belongs to the Sobolev space $W^1_p([0,1];\R^d)$ for $1< p < +\infty$, by \cite[Exercise~7.60]{Friz2010} we obtain that 
  $$
    S(X^m) \in W^1_p([0,1];\R^d)
    \quad \text{with}\quad 
    \|S(X^m)\|_{W^1_p} = \|X^m\|_{W^1_p}.
  $$ 
  This indeed implies that $S(X^m) \in W^\alpha_p([0,1];G^2(\R^d))$.
  
  Alternatively, one can take use of the explicit formula of $\mathbb{X}^m_{\frac{k-1}{2^n},\frac{k}{2^n}}$ and the discrete characterization of Sobolev spaces (see \cite[Theorem~2.2]{Liu2020}) as before to compute that
  \begin{align*}
    \|S(X^m)\|_{W^\alpha_p}^p& \sim \sum_{n=0}^\infty2^{n(\alpha p -1)}\sum_{k=1}^{2^n}d_{cc}(S(X^m)_{\frac{k-1}{2^n}},S(X^m)_{\frac{k}{2^n}})^p \\
    &\le C_m \sum_{n=0}^m2^{n(\alpha p -1)}\sum_{k=1}^{2^n}|X_{\frac{k-1}{2^n}} - X_{\frac{k}{2^n}}|^p \\
    &\le C_m \|X\|_{W^\alpha_p}^p < +\infty,
  \end{align*}
  where the constant $C_m$ may depend on $m$. 
  
  (2) Following the same line as of the proof of Lemma~\ref{lem:convexity of the inhomogeneous discrete norm}, one can check that the functional $F^m$ is strictly convex on $\mathcal{A}(X)$. Moreover, since $S(X^m) \in W^\alpha_p([0,1];G^2(\R^d))$, it holds that $F^m(\X) < +\infty$ for all $\X \in \mathcal{A}(X)$, which means that $F^m$ is proper on $\mathcal{A}(X)$. By the triangle inequality together with the discrete characterization of Sobolev spaces (\cite[Theorem~2.2]{Liu2020}) we see that $F^m$ is coercive and by Fatou's lemma, $F^m$ is lower semi-continuous. Therefore, we conclude that $F^m$ admits a unique minimizer in $\mathcal{A}(X)$ by using Theorem~\ref{thm:existence and uniqueness of optimal extension}.
\end{proof}

If $X$ treated in Lemma~\ref{lemma: optimizer for each dyadic interpolcation} is a typical sample path of Brownian motion, then we can say more about the behaviour of the sequence of minimizers associated to $F^m$. The next proposition says that in this case these minimizers converge to the Stratonovich lift of the Brownian motion with respect to the inhomogeneous Sobolev metric $\hat{\rho}_{W^\alpha_p}$.

\begin{proposition}\label{prop: stratonovich lift can be approx. by minimizers}
  Let $B$ be an $\R^d$-valued Brownian motion. Let $B^m$ denote the dyadic interpolation of $B$ based on the $m$-th dyadic points in $[0,1]$ as before. Let $S(B^m)$ be the canonical lift of $B^m$. Let $\alpha \in (1/3,1/2)$ and $p \in (1,+\infty)$ be such that $\alpha > 1/p$. Then, it holds:
  \begin{enumerate}
    \item For almost all $\omega$, the sequence $S(B^m(\omega))$ converges to the Stratonovich lift $\mathbf{B}^{s}(\omega)$ with respect to the inhomogeneous Sobolev metric $\hat{\rho}_{W^\alpha_p}$. In particular, one almost surely has 
    $$
      \mathbf{B}^s(\omega) \in W^\alpha_p([0,1];G^2(\R^d)).
    $$ 
    \item Let $F^m$ be the functional defined in \eqref{eq:Stratonovich approximation functional} with $S(X^m)$ replaced by $S(B^m)$. Let $\X^{*,m}(\omega)$ denote the unique minimizer of $F^m$ over $\mathcal{A}(B(\omega))$. Then $\X^{*,m}$ converges to the Stratonovich lift $\mathbf{B}^s$ with respect to $\hat{\rho}_{W^\alpha_p}$ almost surely.
  \end{enumerate}
\end{proposition}

\begin{proof}
  (1) For each $m\in \mathbb{N}$ we write $S(B^m) = (1,B^m, \mathbb{B}^m)$. Let us fix an $m\in \mathbb{N}$. First, by \cite[Eq.~(5.4)]{Ledoux2002} we have for all $n \ge m$,
  \begin{equation*}
    \sum_{k=1}^{2^n}\mathbb{E}\Big[\Big|\mathbb{B}^m_{\frac{k-1}{2^n},\frac{k}{2^n}}\Big|^{\frac{p}{2}}\Big] \le C 2^{-n(p-1)}2^{mp/2},
  \end{equation*}
  where the constant $C$ only depends on $p$ and is independent of $n$ and $m$. Using this bound, we deduce that
  \begin{equation*}
    \sum_{n=m}^\infty 2^{n(\alpha p -1)}\sum_{k=1}^{2^n}\mathbb{E}\Big[\Big|\mathbb{B}^m_{\frac{k-1}{2^n},\frac{k}{2^n}}\Big|^{\frac{p}{2}}\Big] \le C 2^{m(\alpha - \frac{1}{2})p},
  \end{equation*}
  and consequently by Jensen's inequality that
  \begin{align*}
    \mathbb{E}\Big[\Big(\sum_{n=m}^\infty 2^{n(\alpha p -1)}\sum_{k=1}^{2^n}\Big|\mathbb{B}^m_{\frac{k-1}{2^n},\frac{k}{2^n}}\Big|^{\frac{p}{2}}\Big)^{\frac{2}{p}}\Big] &\le  \Big(\sum_{n=m}^\infty 2^{n(\alpha p -1)}\sum_{k=1}^{2^n}\mathbb{E}\Big[\Big|\mathbb{B}^m_{\frac{k-1}{2^n},\frac{k}{2^n}}\Big|^{\frac{p}{2}}\Big]\Big)^{\frac{2}{p}} \\
    &\le C 2^{m(2\alpha -1)}.
  \end{align*}
  In particular, we obtain that 
  \begin{align*}
    &\sum_{m=1}^\infty \mathbb{E}\Big[\Big(\sum_{n=m+1}^\infty 2^{n(\alpha p -1)}\sum_{k=1}^{2^n}\Big|\mathbb{B}^{m+1}_{\frac{k-1}{2^n},\frac{k}{2^n}} - \mathbb{B}^{m}_{\frac{k-1}{2^n},\frac{k}{2^n}}\Big|^{\frac{p}{2}}\Big)^{\frac{2}{p}}\Big]  \\ 
    &\,\,\lesssim \sum_{m=1}^\infty \mathbb{E}\Big[\Big(\sum_{n=m+1}^\infty 2^{n(\alpha p -1)}\sum_{k=1}^{2^n}\Big|\mathbb{B}^{m+1}_{\frac{k-1}{2^n},\frac{k}{2^n}}\Big|^{\frac{p}{2}}\Big)^{\frac{2}{p}}\Big] + \sum_{m=1}^\infty \mathbb{E}\Big[\Big(\sum_{n=m+1}^\infty 2^{n(\alpha p -1)}\sum_{k=1}^{2^n}\Big|\mathbb{B}^{m}_{\frac{k-1}{2^n},\frac{k}{2^n}}\Big|^{\frac{p}{2}}\Big)^{\frac{2}{p}}\Big]\\
    & \,\,\lesssim \sum_{m=1}^\infty 2^{m(2\alpha -1)}.
  \end{align*}
  Since $\alpha \in (1/3,1/2)$, we have $2\alpha - 1 < 0$, which implies that
  \begin{equation}\label{eq: the bound for n > m}
    \sum_{m=1}^\infty \mathbb{E}\Big[\Big(\sum_{n=m+1}^\infty 2^{n(\alpha p -1)}\sum_{k=1}^{2^n}\Big|\mathbb{B}^{m+1}_{\frac{k-1}{2^n},\frac{k}{2^n}} - \mathbb{B}^{m}_{\frac{k-1}{2^n},\frac{k}{2^n}}\Big|^{\frac{p}{2}}\Big)^{\frac{2}{p}}\Big] 
    < + \infty.
  \end{equation}
  Let us fix again an $m\in \mathbb{N}$ and consider the difference $\mathbb{B}^{m+1}_{\frac{k-1}{2^n},\frac{k}{2^n}} - \mathbb{B}^{m}_{\frac{k-1}{2^n},\frac{k}{2^n}}$ for $n \le m$. Since we are in the finite dimensional setting, the Euclidean norm on $\R^d \otimes \R^d$ is exact for any Gaussian measure with the parameter equals to $\frac{1}{2}$ (see \cite[Theorem~4]{Ledoux2002}). This result together with the explicit formula of $\mathbb{B}^{m+1}_{\frac{k-1}{2^n},\frac{k}{2^n}} - \mathbb{B}^{m}_{\frac{k-1}{2^n},\frac{k}{2^n}}$ (cf. \cite[Eq.~(5.5)]{Ledoux2002}) gives that
  \begin{equation*}
    \sum_{k=1}^{2^n} \mathbb{E}\Big[\Big|\mathbb{B}^{m+1}_{\frac{k-1}{2^n},\frac{k}{2^n}} - \mathbb{B}^{m}_{\frac{k-1}{2^n},\frac{k}{2^n}}\Big|^{\frac{p}{2}}\Big] \le C 2^{-\frac{mp}{4}}2^{-n(\frac{p}{4}-1)},
  \end{equation*}
  where the constant $C$ only depends on $p$. It follows that
  \begin{align*}
    \sum_{n=0}^m 2^{n(\alpha p -1)}\sum_{k=1}^{2^n} \mathbb{E}\Big[\Big|\mathbb{B}^{m+1}_{\frac{k-1}{2^n},\frac{k}{2^n}} - \mathbb{B}^{m}_{\frac{k-1}{2^n},\frac{k}{2^n}}\Big|^{\frac{p}{2}}\Big] &\le C \sum_{n=0}^m2^{np(\alpha - \frac{1}{4})}2^{-\frac{mp}{4}}\\
    &\le C 2^{mp(\alpha - \frac{1}{2})}.
  \end{align*}
  Then, using Jensen's inequality again, we obtain that
  \begin{align}\label{eq: the bound for n<m}
    \sum_{m=1}^\infty \mathbb{E}\Big[\Big(\sum_{n=0}^m 2^{n(\alpha p -1)}\sum_{k=1}^{2^n}\Big|\mathbb{B}^{m+1}_{\frac{k-1}{2^n},\frac{k}{2^n}} - \mathbb{B}^{m}_{\frac{k-1}{2^n},\frac{k}{2^n}}\Big|^{\frac{p}{2}}\Big)^{\frac{2}{p}}\Big] \lesssim \sum_{m=1}^\infty 2^{m(2\alpha -1)} < +\infty.
  \end{align}
  Combining \eqref{eq: the bound for n > m} and \eqref{eq: the bound for n<m}, we actually obtain almost surely that
  \begin{equation*}
    \sum_{m=1}^\infty \Big(\sum_{n=0}^\infty 2^{n(\alpha p -1)}\sum_{k=1}^{2^n}\Big|\mathbb{B}^{m+1}_{\frac{k-1}{2^n},\frac{k}{2^n}} - \mathbb{B}^{m}_{\frac{k-1}{2^n},\frac{k}{2^n}}\Big|^{\frac{p}{2}}\Big)^{\frac{2}{p}}<+\infty,
  \end{equation*}
  which implies that the sequence $(\mathbb{B}^m)$ is a Cauchy sequence with respect to the inhomogeneous metric $\hat{\rho}^{(2)}_{W^\alpha_p}$. Furthermore, using the Proposition~\ref{prop:weakly geometric is geometric} below we also have that the sequence $(B^m)$ is a Cauchy sequence for the metric $\hat{\rho}^{(1)}_{W^\alpha_p}$. Hence, the sequence $(S(B^m))$ is a Cauchy sequence for $\hat{\rho}_{W^\alpha_p}$.
  
  On the other hand, by \cite[Theorem~3]{Ledoux2002} we know that the sequence $S(B^m)$ converges to the Stratonovich lift $\mathbf{B}^s$ with respect to the inhomogeneous $q$-variation metric for any $q \in (2,3)$ almost surely. This observation indeed implies that the limit of $S(B^m)$ with respect to the Sobolev metric $\hat{\rho}_{W^\alpha_p}$ coincides with $\mathbf{B}^s$ almost surely, where we implicitly use the fact that $W^\alpha_p([0,1];G^2(\R^d))$ is complete with respect to $\hat{\rho}_{W^\alpha_p}$, which can be easily shown.
  
  (2) We consider only those $\omega\in \Omega$ such that $S(B^m(\omega))$ converges to $\mathbf{B}^s(\omega)$, which happens for a set with probability one. Let $\X^{*,m}(\omega)$ denote the unique minimizer of $F^m$ over $\mathcal{A}(B(\omega))$, which exists due to Lemma~\ref{lemma: optimizer for each dyadic interpolcation}. Then by the triangle inequality and the minimization property that $\hat{\rho}_{W^\alpha_p}(\X^{*,m}(\omega), S(B^m(\omega))) = F^m(\X^{*,m}(\omega)) \le F^m(\mathbf{B}^s(\omega)) = \hat{\rho}_{W^\alpha_p}(\mathbf{B}^s(\omega), S(B^m(\omega)))$, we get that
  \begin{align*}
    \hat{\rho}_{W^\alpha_p}(\mathbf{B}^s(\omega), \X^{*,m}(\omega)) &\le \hat{\rho}_{W^\alpha_p}(\mathbf{B}^s(\omega), S(B^m(\omega))) + \hat{\rho}_{W^\alpha_p}(\X^{*,m}(\omega), S(B^m(\omega))) \\
    & \le 2 \hat{\rho}_{W^\alpha_p}(\mathbf{B}^s(\omega), S(B^m(\omega))),
  \end{align*}
  which converges to $0$ as $m \rightarrow \infty$. 
\end{proof}

Here we also record an interesting result that every Sobolev rough path~$\X$ can be approximated by its dyadic geodesic interpolations with respect to the discrete Sobolev distance.

\begin{proposition}\label{prop:weakly geometric is geometric}
  Let $\X \in W^\alpha_p([0,1];G^N(\R^d))$ be a path for $\alpha \in (0,1)$ and $p\in [1,+\infty)$ such that $\alpha>1/p$. For each $n \ge 1$, let $\X^n$ be the finite variation rough path with values in $G^N(\R^d)$ such that $\X^n_t = \X_t$ for all $t = k/2^n$, $k = 0, \ldots, 2^n-1$, and $\X^n\mid_{[k/2^n,(k+1)/2^n]}$ is the geodesic joining $\X_{k/2^n}$ and $\X_{(k+1)/2^n}$. Then, one has 
  \begin{equation*}
    \lim_{n \rightarrow \infty}d_{W^\alpha_p,(1)}(\X^n,\X) = 0.
  \end{equation*}
\end{proposition}

\begin{proof}
  Fix an $n \ge 1$. Since $\X^n$ coincides with $\X$ on all dyadic points of the $n$-th generation, we indeed have
  $$
    \sum_{m=0}^{2^j-1}d_{cc}(\X^n_{m/2^j,(m+1)/2^j},\X_{m/2^j,(m+1)/2^j}) = 0
  $$
  for all $j \le n$. Hence, in view of the definition of $d_{W^\alpha_p,(1)}$ it suffices to show that $\lim_{n \rightarrow \infty} R_n = 0$ for 
  $$
    R_n := \sum_{j=n+1}^\infty 2^{j(\alpha p - 1)}\sum_{m=0}^{2^j-1} d_{cc}(\X^n_{m/2^j,(m+1)/2^j}, \X_{m/2^j,(m+1)/2^j})^p.
  $$
  For each $j \ge n+1$, by triangle inequality one has
  \begin{align*}
    d_{cc}(\X^n_{m/2^j,(m+1)/2^j}, \X_{m/2^j,(m+1)/2^j}) 
    \le & d_{cc}(\X^n_{m/2^j}, \X^n_{(m+1)/2^j}) + d_{cc}(\X_{m/2^j}, \X_{(m+1)/2^j}).
  \end{align*}
  Since $\X \in W^\alpha_p([0,1];G^N(\R^d))$, the condition $\|\X\|_{W^\alpha_p,(1)} <+ \infty$ ensures that 
  $$
    \lim_{n \rightarrow \infty} \sum_{j=n+1}^\infty 2^{j(\alpha p - 1)}\sum_{m=0}^{2^j-1} d_{cc}(\X_{m/2^j}, \X_{(m+1)/2^j})^p = 0.
  $$ 
  Thus in order to get $\lim_{n \rightarrow \infty}R_n = 0$, we only need to show that
  \begin{equation}\label{eq: asymptotic behaviour of geodesic interpolation}
    \lim_{n \rightarrow \infty} \sum_{j=n+1}^\infty 2^{j(\alpha p - 1)}\sum_{m=0}^{2^j-1} d_{cc}(\X^n_{m/2^j}, \X^n_{(m+1)/2^j})^p = 0.
  \end{equation}
  For this purpose, we note that since $\X^n$ is the geodesic linking $\X_{m/2^n}$ with $\X_{(m+1)/2^n}$, $m = 0, \ldots, 2^n - 1$, for $j \ge n+1$, $k/2^j \in [m/2^j,(m+1)/2^j]$, it holds that 
  $$
    d_{cc}(\X^n_{k/2^j}, \X^n_{(k+10/2^j)}) = 2^{n-j}d_{cc}(\X_{m/2^n}, \X_{(m+1)/2^n}).
  $$
  As a consequence, we can readily deduce that  
  \begin{align*}
    \sum_{j=n+1}^\infty  2^{j(\alpha p - 1)}
    &\sum_{m=0}^{2^j-1} d_{cc}(\X^n_{m/2^j}, \X^n_{(m+1)/2^j})^p \\
    &= \Big(\sum_{j = n+1}^\infty 2^{(n-j)(1-\alpha)p}\Big) 
    2^{n(\alpha p -1)} \sum_{m=0}^{2^n-1} d_{cc} (\X_{m/2^n},\X_{(m+1)/2^n})^p.
  \end{align*}
  Since $\sum_{j = n+1}^\infty 2^{(n-j)(1-\alpha)p} \sim 2^{(1-\alpha)p}$, letting $n$ tend to infinity leads to~\eqref{eq: asymptotic behaviour of geodesic interpolation}.
\end{proof}

Finally, we have all ingredients at hand to prove Theorem~\ref{thm:Stratonovich is optimal}.

\begin{proof}[Proof of Theorem~\ref{thm:Stratonovich is optimal}]
  By Proposition~\ref{prop: stratonovich lift can be approx. by minimizers} the Stratonovich rough path lift $\mathbf{B}^s$ belongs to $W^{\alpha}_p([0,1];G^2(\R^d))$ and is a solution to the optimization problem~\eqref{eq:Stratonovich is optimal} with $F(\mathbf{B}^s)=0$ almost surely. Let $\omega\in \Omega$ be such that $F(\mathbf{B}^s(\omega))=0$.
  
  It remains to show the uniqueness of $\mathbf{B}^s(\omega)$ as optimizer. To this end, suppose $\X^*$ is another optimizer of~\eqref{eq:Stratonovich is optimal} given $B(\omega)$. In this case we have 
  \begin{align*}
    0 \leq \hat{\rho}_{W^\alpha_p} (\mathbf{B}^s(\omega),\X^*)
    \leq \hat{\rho}_{W^\alpha_p} (\mathbf{B}^s(\omega),S(B^m(\omega))) 
    + \hat{\rho}_{W^\alpha_p} (S(B^m(\omega)),\X^*).
  \end{align*}
  Taking the limes superior, we get $\hat{\rho}_{W^\alpha_p} (\mathbf{B}^s(\omega),\X^*)=0$, which implies the claimed uniqueness. 
\end{proof}

\subsection{Optimal extension of lower dimensional rough paths}

A related extension problem of interest is whether there exists a optimal joint rough path lift given two lower dimensional weakly geometric rough paths. For instance, this task appears in the context of rough differential equations with mean field interaction, see \cite{Cass2015}, or in the context of robust non-linear filtering, see \cite{Diehl2015}.

More precisely, for $k,l\in \N$ and for two Sobolev rough paths
$$
  \X\in W^{\alpha}_p([0,T];G^{[\frac{1}{\alpha}]}(\mathbb{R}^k)) \quad \text{and} \quad  \Y \in W^{\alpha}_p([0,T];G^{[\frac{1}{\alpha}]}(\mathbb{R}^l)),
$$ 
with $d:=k+l$, define the \textit{admissible set}  
\begin{align*}
  \mathcal{A}(\X,\Y)
  := \Big\{ \Z \in W^{\alpha}_p([0,T];G^{[\frac{1}{\alpha}]}(\mathbb{R}^d))\,:\, \pi_{d,k} (\Z) =\X\text{ and }\pi_{d,l} (\Z) =\Y \Big\},
\end{align*}
where $\pi_{d,k}\colon G^{[\frac{1}{\alpha}]}(\R^d)\to G^{[\frac{1}{\alpha}]}(\R^k)$ is the extension of the canonical projection $p_{d,k}\colon \R^d\to \R^k$ and $\pi_{d,l}$ is analogously defined. For more details we refer to \cite[Section~7.5.6]{Friz2010} and, in particular, \cite[Example~7.54]{Friz2010}. 

\begin{problem}\label{prob:joint lift problem}
  Given two lower-dimensional Sobolev rough paths $\X\in W^{\alpha}_p([0,T];G^{[\frac{1}{\alpha}]}(\R^k))$ and $\Y \in W^{\alpha}_p([0,T];G^{[\frac{1}{\alpha}]}(\mathbb{R}^l))$ and a selection criterion $F\colon \mathcal{A}(\X,\Y) \to \R \cup \{+\infty\}$, we are looking for an admissible rough path lift $\Z^*\in \mathcal{A}(\X,\Y)$ such that
  \begin{equation*}
    F(\Z^*)=\min_{\Z\in \mathcal{A}(\X,\Y)} F(\Z).
  \end{equation*}
\end{problem}

Indeed, we will show that if $\alpha \in (1/3,1/2)$ (that is, if $[\frac{1}{\alpha}] = 2$) and if the functional $F$ satisfies all assumptions required in Theorem~\ref{thm:existence and uniqueness of optimal extension}, the optimal extension problem looking for joint rough path lifts of two lower dimensional rough paths described as in the Optimal Extension Problem~\ref{prob:joint lift problem} admits a (unique) solution. 

\begin{proposition}\label{prop:optimization of 2 levels joint lifts}
  Let $\X\in W^{\alpha}_p([0,T];G^{2}(\mathbb{R}^k))$ and $\Y \in W^{\alpha}_p([0,T];G^{2}(\mathbb{R}^l))$ be given for $\alpha \in (\frac{1}{3},\frac{1}{2})$. Then, the admissible set $\mathcal{A}(\X,\Y)$ is non-empty. Moreover, let $F$ be a proper functional defined on $\mathcal{A}(\X,\Y)$ such that $F$ admits an extension $\overline{F}$ defined on the Banach  $W^\alpha_p([0,T];\R^d \oplus (\R^d)^{\otimes2})$ which is coercive, convex, lower semi-continuous and equals to $+\infty$ outside $\mathcal{A}(\X,\Y)$. Then, the Optimal Extension Problem~\ref{prob:joint lift problem} admits at least one solution~$Z^*$. If in addition $\overline{F}$ is strictly convex, then the solution~$Z^*$ is unique.
\end{proposition}

\begin{proof}
  First let us prove that $\mathcal{A}(\X,\Y)$ is always non-empty. Let $\Z_t := \exp\big(\log \X_t + \log \Y_t\big)$ for $t \in [0,T]$. Here one needs to interpret $\log \X_t$ as a vector in the bigger Lie algebra $\mathcal{G}^2(\R^d)$ by identifying $\mathcal{G}^2(\R^k)$ as a subspace of $\mathcal{G}^2(\R^d)$ in the natural way, and the same interpretation also holds for $\log \Y_t$. Then the path $\Z$ takes values  in the quotient group $G^2(\R^d)/L$, where $L = \exp \mathcal{L}$ and $\mathcal{L}$ is the Lie ideal generated by 
  $$
    [\R^k, \R^l] = \text{span}\{e_i \otimes e_{k+j} - e_{k+j} \otimes e_i: i=1, \dots , k, j=1, \dots , l\}
  $$
  where $e_{m}$, $m=1,\dots,d$, is the canonical basis of $\R^d = \R^k \oplus \R^l$. Using the Campbell--Baker--Hausdorff formula one can verify that for $s,t \in [0,T]$, 
  $$
    \Z_{s,t} = \Z_s^{-1} \otimes \Z_t = \X_{s,t} \otimes \Y_{s,t} \otimes \exp{M_{s,t}},
  $$
  where $M_{s,t}$ belongs to $\mathcal{L}$, and we interpret $\X_{s,t}$ and $\Y_{s,t}$ as elements in $G^2(\R^d)$ by identifying $G^2(\R^k)$ and $G^2(\R^l)$ as Lie subgroups of $G^2(\R^d)$ in the natural way. Hence, we have
  $$
    \|\Z_{s,t}\|_{G^2(\R^d)/L} = \inf_{l \in L}\|\Z_{s,t}\otimes l\| \le \|\Z_{s,t} \otimes \exp(-M_{s,t})\| = \|\X_{s,t} \otimes \Y_{s,t}\|.
  $$
  Since $\X$ and $\Y$ have finite Sobolev norms, the sub-additivity of Carnot--Caratheodory norm $\| \cdot \|$ provides that $\Z \in W^\alpha_p([0,T];G^2(\R^d)/L)$. Then, by Lyons--Victoir extension theorem for Sobolev paths (see Theorem~\ref{thm:Generalization of Theorem 14 in Lyons-Victoir})  we can lift $\Z$ to a $\tilde{\Z} \in W^\alpha_p([0,T];G^2(\R^d))$ such that $\pi_{G^2(\R^d),G^2(\R^d)/L}(\tilde{\Z}) = \Z$. Thanks to the definition of $\Z$ we have $\pi_{d,k}(\tilde{\Z}) = \X$ and $\pi_{d,l}(\tilde{\Z}) = \Y$, and thus we obtain that $\tilde{\Z} \in \mathcal{A}(\X,\Y)$.
  
  Now the existence and uniqueness of a solution~$\Z^*$ to the Optimal Extension Problem~\ref{prob:joint lift problem} follows immediately from Theorem~\ref{thm:existence and uniqueness of optimal extension} because in this case one can check that $\mathcal{A}(\X,\Y)$ is a convex subset of $\mathcal{A}(X)\subset W^\alpha_p([0,T];\R^d \oplus (\R^d)^{\otimes2})$, where $X := \pi_1(\X) \oplus \pi_1(\Y)$ is an element of the quotient group $G^2(\R^d)/\exp (W_2) \cong \R^d$ and the convexity is the usual one defined on the Banach space $W^\alpha_p([0,T];\R^d \oplus (\R^d)^{\otimes2})$.
\end{proof}

\begin{remark}
  Analogously to the proof of Proposition~\ref{prop:optimization of 2 levels joint lifts}, one can actually show that $\mathcal{A}(\X,\Y)$ is non-empty for all $\alpha \in (0,1)$ and $p\in (1,+\infty)$ with $\alpha > \frac{1}{p}$, using Lyons--Victoir extension theorem for Sobolev paths (see Theorem~\ref{thm:Generalization of Theorem 14 in Lyons-Victoir}).
\end{remark}

\section{Discussion: generalizations to Besov spaces}\label{sec:Besov spaces}

In principle, all results presented in the previous sections extend from Sobolev spaces to the even more general class of Besov spaces. Since this would make the paper much longer and technically more involved without (most likely) leading to additionally conceptional new insights, we decided to write the paper in the Sobolev setting. However, some results extend immediately to a Besov topology without extra conceptional effort. Thus an extension might be interesting in view of the Besov rough analysis recently provided in~\cite{Friz2021}. In this section we shall point out which ones.

\smallskip

We start by recalling the definition of (non-linear) Besov spaces. Let $(E,d)$ be a metric space. For $\alpha\in (0,1)$ and $p,q\in [1,+\infty]$ the \textit{Besov space} $B_{p,q}^\alpha([0,1];E)$ consists of all measurable functions $f\colon [0,1] \to E $ such that 
\begin{equation*}
  {\|f\|}_{B^\alpha_{p,q}} :=\bigg(\int_{0}^1 \Big(\int_0^{1-h} \frac{d(f(x),f(x+h))^p}{h^{\alpha p}}\dd x\Big)^{\frac{q}{p}}\frac{\dd h}{h} \bigg)^{1/q}<+\infty.
\end{equation*}
Note that $B_{p,p}^\alpha([0,1];E)$ coincides with the Sobolev space $W_{p}^\alpha([0,1];E)$. This leads naturally to the notion of Besov rough paths. 

\begin{definition}
  Let $0<\alpha<1$ and $1\le p,q\le + \infty$ such that $\alpha  > 1/p$. Every $\X\in B^\alpha_{p,q}([0,1];G^{[\frac{1}{\alpha}]}(\R^d))$ is called \textit{Besov rough path} with regularity $(\alpha,p,q)$. 
\end{definition}

The main ingredient for some results (pointed out below) was the equivalence of ${\|f\|}_{B^\alpha_{p,q}}$ and
\begin{equation*}
  {\|f\|}_{b^\alpha_{p,q},(1)} :=\bigg(  \sum_{j \geq 0} 2^{jq(\alpha - \frac{1}{p})} \Big(\sum_{m=0}^{2^j-1} { d\big(f\big(\frac{m+1}{2^j}),f\big(\frac{m}{2^j}\big) \big)  }^p\Big)^{\frac{q}{p}}\bigg)^{1/q}
\end{equation*}
for continuous functions $f\colon [0,1] \to E$, see \cite[Theorem~2.2]{Liu2020}, which we so far only use in the Sobolev case. Based on this equivalence, one can easily modify the poof of Lemma~\ref{lem:Generalization of Lemma 13 in Lyons-Victoir} to obtain its Besov analogue and follow the lines of \cite{Lyons2007a} to arrive at the \textit{Lyons-Victoir theorem for Besov paths}, cf. Corollary~\ref{cor:Generalization of Corollary 19 in Lyons-Victoir}. Let $0<\alpha<1$ and $1\le p,q\le + \infty$ be such that $\alpha > 1/p$ such that $\frac{1}{\alpha} \notin \N\setminus \{1\}$ for the rest of the section. 

\begin{proposition}
  Then, every $\R^d$-valued path $X\in B^\alpha_{p,q}([0,1];\R^d)$ can be lifted to a Besov rough path in $B^\alpha_{p,q}([0,1];G^{[\frac{1}{\alpha}]}(\R^d))$.
\end{proposition}

The Optimal Extension Problem~\ref{prob:optimal extension} can also be considered for Besov rough paths. For $K$  a closed subgroup of $\exp(W_{[\frac{1}{\alpha}]})$ and $X\in B^{\alpha}_{p,q}([0,1];G^{[\frac{1}{\alpha}]}(\R^d)/K)$ we only need to replace the admissible set $\mathcal{A}(X)$ by
\begin{equation*}
  \mathcal{A}_B(X):= \Big\{ \X \in B^{\alpha}_{p,q}([0,1];G^{[\frac{1}{\alpha}]}(\mathbb{R}^d))\,:\, \pi_{G^{[\frac{1}{\alpha}]}(\R^d),G^{[\frac{1}{\alpha}]}(\R^d)/K} (\X) =X\Big\}.
\end{equation*}

Following the proof of Lemma~\ref{lem:admissible set is convex} and Theorem~\ref{thm:existence and uniqueness of optimal extension}, we deduce without difficulties:

\begin{proposition}
  Let $F \colon \mathcal{A}_B(X) \rightarrow \R \cup \{+\infty\}$ be a proper functional defined on $\mathcal{A}_B(X)$, that is, $F(\X) \in \R$ holds for at least one $\X \in \mathcal{A}_B(X)$. If $F$ can be extended to a coercive, strictly convex and lower semi-continuous functional $\overline{F}$ defined on the Banach space $B^\alpha_{p,q}([0,1]; \bigoplus_{i=1}^{[\frac{1}{\alpha}]} (\R^d)^{\otimes i})$ such that $\overline{F} = +\infty$ outside $\mathcal{A}(X)$, then the Optimal Extension Problem~\ref{prob:optimal extension} admits a unique solution~$\X^*$ in $\mathcal{A}_B(X)$. 
\end{proposition}

\medskip

\noindent{\bf Data availability statement:} Data sharing is not applicable to this article as no datasets were generated or analysed during the current study.


\providecommand{\bysame}{\leavevmode\hbox to3em{\hrulefill}\thinspace}
\providecommand{\MR}{\relax\ifhmode\unskip\space\fi MR }
\providecommand{\MRhref}[2]{%
  \href{http://www.ams.org/mathscinet-getitem?mr=#1}{#2}
}
\providecommand{\href}[2]{#2}

\end{document}